%% file: diplomarbeit.tex
\documentclass[pdftex, a4paper, 12pt]{article}
\pdfoutput=1
\newcommand{\Philipp}{}
\input{chapters/header}

\author{Philipp Harms, Matrikelnummer 0326036}
\title{The Poincar\'{e} Lemma in Subriemannian Geometry}

\begin{document}

\input{chapters/titlepagephilipp}
\input{chapters/introductionphilipp}

\tableofcontents
\newpage

\input{chapters/preliminaries}

\input{chapters/firstproofofchow}

\input{chapters/secondproofofchow}
\input{chapters/thirdproofofchow}
\input{chapters/poincarelemma}

\bibliographystyle{plainnat}
\bibliography{bib/bibliography}
\printindex
\end{document}

%% file: chapters/header.tex
\usepackage[width=15cm, height=23.5cm,hmarginratio=4:3]{geometry}
\usepackage[latin1]{inputenc}
\usepackage{amsmath, amssymb, amsthm, amscd, emp, ifpdf, graphicx, ifthen, enumerate}
\ifpdf
\fi

\usepackage{makeidx}
\makeindex

\usepackage[numbers]{natbib}
\usepackage{url}
\usepackage{color}
\usepackage[bookmarks=true, 
	bookmarksopen=true,
         bookmarksnumbered=true, 
         colorlinks=true,
         citecolor = black,
         filecolor = black,
         linkcolor = black,
         urlcolor  = black,
         plainpages=false,
         hyperindex=true]{hyperref}
         
\ifthenelse{\isundefined \Philipp}{
\hypersetup{pdftitle={Subriemannian Geometry},
	pdfauthor={Martin Bauer, Matrikelnummer 0325021},
	pdfcreator={pdfLaTeX},
         pdfproducer={LaTeX mit hyperref}}
}{
\hypersetup{pdftitle={Subriemannian Geometry},
	pdfauthor={Philipp Harms, Matrikelnummer 0326036},
	pdfcreator={pdfLaTeX},
         pdfproducer={LaTeX mit hyperref}}
}

\theoremstyle{plain}
\newtheorem{corollary}{Corollary}[section]
\newtheorem{theorem}[corollary]{Theorem}
\newtheorem{lemma}[corollary]{Lemma}

\theoremstyle{definition}
\newtheorem{definition}[corollary]{Definition}
\theoremstyle{remark}
\newtheorem{remark}[corollary]{Remark}
\newtheorem{example}[corollary]{Example}

\DeclareMathOperator{\Evol}{Evol}
\DeclareMathOperator{\MyBox}{Box}
\DeclareMathOperator{\End}{End}
\DeclareMathOperator{\Id}{Id}
\DeclareMathOperator{\linearspan}{span}
\DeclareMathOperator{\aut}{aut}
\DeclareMathOperator{\sgn}{sgn}

\newcommand{\ud}{\mathrm{d}}
\newcommand{\norm}[1]{\left\|{#1}\right\|}
\newcommand{\abs}[1]{\left|{#1}\right|}
\newcommand{\evaluate}[2]{\left. {#1} \right|_{#2}}
\newcommand{\innerproduct}[2]{\left\langle #1 , #2 \right\rangle}



%% file: chapters/titlepagephilipp.tex
%
%

\begin{titlepage}
\begin{center}

\begin{figure}[h]
\begin{center}
\includegraphics{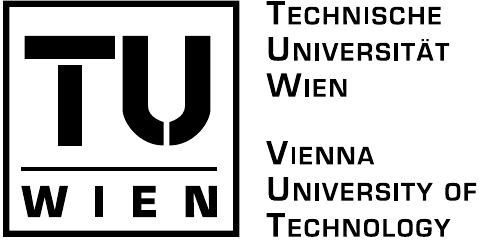}
\end{center}
\end{figure}
\vspace{2cm}

{\large D I P L O M A R B E I T} \vspace{1.5cm}

{\Huge The Poincar\'{e} Lemma \\
\vspace{0.3cm} in Subriemannian Geometry} \vspace{1.5cm}

ausgef\"{u}hrt am Institut f\"{u}r \\ \vspace{2mm}
{\large Wirtschaftsmathematik}\\ \vspace{2mm}
der Technischen Universit\"{a}t Wien \\ \vspace{1cm}
unter Anleitung von {\large Ao.Univ.Prof. Mag. Dr. Josef Teichmann}\\\vspace{2.5cm}
durch \\ \vspace{2mm}
{\large Philipp Harms}\\ \vspace{2mm}
Josefst\"adterstra\ss{}e 14/54 \\
1080 Wien \vspace{3cm}

\begin{tabular}
[c]{ccc}
\underline{\hspace*{5.2cm}} & \hspace*{3.5cm} & \underline{\hspace*{5.2cm}}\\
         Datum              &                 &         Unterschrift
\end{tabular}
\end{center}
\end{titlepage}

\newpage \mbox{ } \thispagestyle{empty}

\vspace{18cm}
\noindent 
\begin{center}
I am very grateful to my advisor Prof. Josef Teichmann, who has recommended the subject of subriemannian geometry to me, for his time and ongoing support. 

\vspace{.5cm}

I would like to thank Prof. Gruber for the integration in his institute. 

\vspace{.5cm}

I had the most wonderful time writing this work with my colleague Martin Bauer. 
\end{center}

\newpage
\setcounter{page}{1}

%% file: chapters/introductionphilipp.tex
%
%

\section*{Abstract}

This work is a short, self-contained introduction to subriemannian geometry with special emphasis on Chow's Theorem. As an application, a regularity result for the Poincar\'e Lemma is presented. 

At the beginning, the definitions of a subriemannian geometry, horizontal vectorfields and horizontal curves are given. Then the question arises: Can any two points be connected by a horizontal curve? Chow's Theorem gives an affirmative answer for bracket generating distributions. (A distribution is called bracket generating if horizontal vectorfields and their iterated Lie brackets span the whole tangent space.) 

We present three different proofs of Chow's Theorem; each one is interesting in its own. The first proof is based on the theory of Stefan and Sussmann regarding integrability of singular distributions. The second proof is elementary and gives some insight in the shape of subriemannian balls. The third proof is based on infinite dimensional analysis of the endpoint map. 

Finally, the study of the endpoint map allows us to prove a regularity result for the Poincar\'e Lemma in a form suited to subriemannian geometry: If for some $r \geq 0$ all horizontal derivatives of a given function $f$ are known to be $r$ times continuously differentiable, then so is $f$. 

Sections~\ref{sec:subriemanniangeometry}~to~\ref{sec:chow} are the common work of Martin Bauer and Philipp Harms. 

\newpage

%% file: chapters/preliminaries.tex
%
%

\section{The definition of a subriemannian geometry}\label{sec:subriemanniangeometry}

For the basic definitions in differential geometry, see for example the book of \citet{Jaenich}. A more detailed treatment, including calculus in Banach spaces and Banach manifolds, can be found in the book of \citet{Abraham}. A good introduction to subriemannian geometry is the book of \citet{Montgomery}.
Throughout this paper, manifolds are assumed to be paracompact and smooth. Vectorfields and mappings between manifolds are assumed to be smooth if not stated otherwise. 

\begin{definition}\label{distribution} 
\index{distribution} \index{horizontal}
Let $M$ be a manifold. Suppose that for each $x \in M$ we are given a sub vector space $H_x$ of $T_x M$. The disjoint union $H=\coprod_{x\in M} H_x$ is called a distribution on $M$. We call tangent vectors in $H$ horizontal. 
\end{definition}

\begin{definition}\index{regular distribution} \index{singular distribution}
A distribution $H$ is called regular if the dimension of $H_x$ is locally constant, and singular otherwise. 
\end{definition}

\begin{definition}\index{vectorfield}\index{vectorfield ! local} \index{local vectorfield}
Let $\mathfrak{X}(M)$ denote the set of all vectorfields on $M$, and let $\mathfrak{X}_{loc}(M)$ be the set of local vectorfields, i.e. 
\begin{equation*}
\mathfrak{X}_{loc}(M) = \bigcup \left\{ \mathfrak{X}(U): U \text{ is an open subset of } M \right\}.
\end{equation*}
To avoid the need for remarks such as ``provided the domains of the vectorfields intersect'', we shall declare the vectorfield defined on the empty set to be an element of $\mathfrak{X}_{loc}(M)$. 
\end{definition}

\begin{definition}\index{vectorfield ! horizontal} \index{horizontal ! vectorfield}
A local vectorfield is called horizontal if it is horizontal at every point where it is defined. Let $\mathfrak{X}_H(M)$ denote the set of horizontal vectorfields and $\mathfrak{X}_{loc, H}(M)$ the set of local horizontal vectorfields  on $M$. 
\end{definition}

\begin{definition}\index{span}\index{distribution ! smooth} \index{smooth distribution}
We say that a set of local vectorfields $\mathcal V \subset \mathfrak{X}_{loc}(M)$ spans a distribution $H$ if for all points $x \in M$, $H_x$ is the linear hull of the set $\left\{X(x):X\in \mathcal V \right\}$. We say that a distribution $H$ is a smooth if it is spanned by $\mathfrak{X}_{loc,H}(M)$. 
\end{definition}

\begin{remark}
Every set of vectorfields spans a distribution, which is obviously smooth (the linear hull of the empty set is the vector space $\{0\}$). 
\end{remark}

\begin{definition}\index{locally finitely generated distribution}\index{finitely generated distribution}\index{distribution ! locally finitely generated}
A distribution $H$ is called locally finitely generated if for every point $x \in M$ there is neighbourhood $U$ of $x$ and a finite set of vectorfields $X_1, \dots, X_k \in \mathfrak X_H(U)$ spanning $H$ on $U$. 
\end{definition}

\begin{definition}\index{subriemannian metric}
Let $H$ be a smooth distribution on $M$. A subriemannian metric $\innerproduct{\cdot}{\cdot} $ on $H$ is a map that assigns to each point $x \in M$ an inner product $\innerproduct{\cdot}{\cdot}_x$ on $H_x$. Let $\norm{X}_x$ designate $\sqrt{\innerproduct{X}{X}_x}$ for $X \in H_x$. We say that a subriemannian metric is smooth if for all local horizontal vectorfields $X$ and $Y$, $x \mapsto \innerproduct{X(x)}{Y(x)}_x$ is smooth. 
\end{definition}

The smoothness of $\left\langle \cdot,\cdot \right\rangle$ does not imply that $\norm{\cdot}: H\rightarrow \mathbb{R}$ is continuous, when $H$ is endowed with the subspace topology of $TM$. 

\begin{example}
To give an example, let $H$ be the distribution on $\mathbb{R}$ spanned by the vectorfield $x\frac{\partial}{\partial x}$, and let $\left\langle \cdot,\cdot \right\rangle$ be such that 
$$\innerproduct{\frac{\partial}{\partial x}}{\frac{\partial}{\partial x}}_x=\frac{1}{x} \quad \text{for} \quad x \neq 0.$$ 
Then $\left\langle \cdot,\cdot \right\rangle$ is smooth, but is not continuous, where $H$ is given the subspace topology. To see that $\left\langle \cdot,\cdot \right\rangle$ is smooth, let $X_1,X_2$ be horizontal vectorfields defined near zero, $X_i(x)=f_i(x)\evaluate{\frac{\partial}{\partial x}}{x}$ with $f_i$ smooth, $f_i(0)=0$. Then
\begin{multline*}
\left\langle X_1(x),X_2(x) \right\rangle_x=\frac{1}{x}f_1(x)f_2(x)= 
\frac{1}{x}\int_0^1{\frac{\ud}{\ud s}(f_1(xs)f_2(xs))\ud s} = \\ 
= \int_0^1(f_1'(xs)f_2(xs)+f_1(xs)f_2'(xs))\ \ud s \in C^{\infty}.
\end{multline*}
To see that $\|\cdot\|$ is not continuous, take $X_n=\frac{1}{\sqrt n}\evaluate{\frac{\partial}{\partial x}}{\frac{1}{n}}$. Then $X_n$ converges to the zero tangent vector at $x=0$, but $\|X_n\|=1$.
\end{example}

\begin{definition}\index{subriemannian manifold} 
A subriemannian manifold is a manifold endowed with a smooth distribution and a smooth subriemannian metric on this distribution. 
\end{definition}

From now on, we will only consider smooth distributions and smooth subriemannian metrics. 

\begin{definition} \label{def:absolutelycontinuouscurve}\index{curve ! absolutely continuous}\index{absolutely continuous curve}
A curve $\gamma$ is called absolutely continuous if it is absolutely continuous in any chart, i.e. for any chart $(U,u)$ of $M$ and any interval $[t_0,t_1]$ such that $\gamma([t_0,t_1]) \subset U$, the curve $u \circ \gamma | [t_0,t_1]$ is absolutely continuous. 
\end{definition}

\begin{definition}\label{def:controlledcurve}\index{curve ! controlled} \index{controlled curve} \index{control}
A curve $\gamma:[a,b] \rightarrow M$ is said to be controlled by local vectorfields $X_1, \dots, X_k$ defined near $\gamma([a,b])$ if $\gamma$ is absolutely continuous and if there are $L^1$-functions $u_1, \dots, u_k$ such that the equation $$\dot\gamma(t)=\sum_{i=1}^k u_i(t)\ X_i(\gamma(t))$$ holds almost everywhere. $u_1, \dots, u_k$ are called the controls of $\gamma$ with respect to the vectorfields $X_1, \dots, X_k$. We also say that the controls steer the point $\gamma(a)$ to the point $\gamma(b)$.  
\end{definition}

\begin{definition}
\label{def:horizontalcurve}\index{curve ! horizontal}\index{horizontal ! curve}
A curve $\gamma$ is called horizontal if it is a concatenation of curves that are controlled by horizontal vectorfields.
\end{definition}

\begin{definition}
\label{def:accessibleset}\index{accessible set}\index{pathwise connected} \index{horizontally pathwise connected}
For a point $x \in M$, the accessible set $Acc(x)$ is the set of all points which can be reached by a horizontal curve starting at $x$. A set $S \subset M$ is called horizontally pathwise connected if any two points in $S$ can be connected by a horizontal curve lying in $S$. 
\end{definition}

\begin{lemma}
If $\gamma$ is a horizontal curve in a subriemannian manifold, then the function $t \mapsto \norm{\dot\gamma(t)}_{\gamma(t)}$ is integrable.
\end{lemma}

\begin{proof}
A horizontal curve is a concatenation of curves controlled by horizontal vectorfields. If $\gamma$ is controlled by horizontal vectorfields $X_1,\dots,X_k$ with $L^1$-controls $u_1, \dots, u_k$, we have
\begin{equation*} \begin{split}
\norm{\dot\gamma}_{\gamma(t)}^2 & = \sum_{i=1}^k \sum_{j=1}^k \abs{u_i(t)} \abs{u_j(t)} \abs{\innerproduct{X_i(\gamma(t))}{X_j(\gamma(t))}_{\gamma(t)}} \\ 
& \leq K \left( \sum_{i=1}^k \abs{u_i(t)} \right) \left( \sum_{j=1}^k \abs{u_j(t)} \right) = K \left( \sum_{i=1}^k \abs{u_i(t)} \right)^2, 
\end{split}\end{equation*}
where $K$ bounds the continuous functions 
\begin{equation*}
t \mapsto \abs{\innerproduct{X_i(\gamma(t))}{X_j(\gamma(t))}_{\gamma(t)}} , \quad 1 \leq i,j \leq k. \qedhere
\end{equation*}
\end{proof}

\begin{definition} \label{def:lengthofacurve}\index{length of a curve}\index{curve ! length}
If $M$ is a subriemannian manifold, we define the length of a horizontal curve $\gamma:[a,b]\rightarrow M$ by
\begin{equation*}
l(\gamma) = \int_a^b \norm{\dot\gamma(t)}_{\gamma(t)} \ud t.
\end{equation*}
\end{definition}

\begin{definition}\label{def:subriemanniandistance}\index{subriemannian distance}\index{Carnot-Caratheodory distance}
The subriemannian distance, also called  Carnot-Carath\'{e}odory distance, between two points $x$ and $y$ on a subriemannian manifold is given by
\begin{equation*} d(x, y) = \inf l(\gamma), \end{equation*}
where the infimum is taken over all horizontal curves that connect $x$ to $y$. The distance is infinite if there is no such curve. The subriemannian ball of radius $\epsilon$ centered at $x \in M$ is denoted by
\begin{equation*} B(\epsilon, x) = \{y \in M : d(x, y) < \epsilon \}. \end{equation*}
\end{definition}

\begin{definition}\index{subriemannian ball}
A horizontal curve $\gamma$ that connects $x$ to $y$ is called a geodesic if $l(\gamma)=d(x,y)$. 
\end{definition}

\subsection{Accessibility and the set of horizontal curves}\index{accessible set}

We have defined the accessible set as the set of points reachably by a certain class of curves, namely the curves as in Definition \ref{def:horizontalcurve}. Theorems \ref{thm:accessiblesetisintegralmanifold} and \ref{thm:accessiblesetandendpointmap} confirm that our choice of admissible curves is a good one. It can be seen in the proofs of these theorems that the important point is that we require a curve to be controlled by horizontal vectorfields. \citet{Bellaiche} shows that instead of $L^1$-controls, we can use any reasonable smaller set of controls without changing the notion of accessibility:

\begin{theorem} 
Let $C$ be a dense subspace in $L^1([0,1];\mathbb R^k)$. Then any point accessible from $x$ by means of controls in $L^1([0,1];\mathbb R^k)$ is also accessible from $x$ by means of controls in $C$. 
\end{theorem}

$C$ can even be the set of piecewise constant functions as will be seen in Section \ref{sec:orbits} about orbits. 

It is a different question whether we can enlarge the class of horizontal curves without changing the notion of accessibility. For example it would be tempting to drop the concept of controllability and instead define a curve to be horizontal if it is absolutely continuous and if its derivative lies in $H$ whenever it exists. This works out for regular distributions as will be shown in Lemma \ref{lem:horizontalhorizontal} below, but we will see in the following example that generally this might change the notion of accessibility. 

\begin{example}
Let $H$ be the distribution on $\mathbb R$ that has rank zero at points $x \leq 0$ and rank one at points $x>0$. Then $H$ is smooth. It is easy to construct an absolutely continuous curve $\gamma$ with $\gamma(0)=0$ and $\gamma(t)>0$ for $t>0$ such that $\dot\gamma(0)$ either vanishes or does not exist. But such a curve is not controllable by any horizontal vectorfield because controllability implies that the curve is constant. To show this, let us assume that $\gamma$ is controlled by a horizontal vectorfield $X$. Then $\gamma$ satisfies the differential equation
$$\dot\gamma(t)=u(t)\ X(\gamma(t)), \quad \gamma(0) =0$$
with $u \in L^1([0,1];\mathbb R)$. By Theorem \ref{thm:controlledcurve} the above differential equation has a unique solution for small $t$. But also the constant curve is a solution to the above equation, and therefore $\gamma$ is constant. 
\end{example}

\begin{lemma}\label{lem:finitely generated}
A regular, smooth distribution is locally finitely generated, and it is a vector subbundle of the tangent bundle. 
\end{lemma}

\begin{proof}
Fix an arbitrary point $x \in M$, and let $k$ be the rank of the distribution $H$ near $x$. Because $H$ is smooth, there are local horizontal vectorfields $X_1, \dots, X_k$ such that $X_1(x), \dots, X_k(x)$ span $H_x$. $X_1, \dots, X_k$ are linearly independent also in some neighbourhood of $x$, so they span $H$ in this neighbourhood, and therefore $H$ is locally finitely generated. 

Now take arbitrary vectorfields $X_{k+1}, \dots, X_n$ defined near $x$ such that the tangent vectors $X_1(x), \dots, X_n(x)$ are linearly independent. Then $X_1, \dots, X_n$ are linearly independent also in some neighbourhood of $x$, and so they form a local frame of the tangent space near $x$. To every tangent vector in the tangent space near $x$, we can associate its coordinates in terms of this local frame. This yields a local trivialization for $H$, and so $H$ is a subbundle of $TM$. 
\end{proof}

\begin{lemma} \label{lem:riemannianmetric}
If $\innerproduct{\cdot}{\cdot}$ is a smooth subriemannian metric on a regular smooth distribution $H$, then $\innerproduct{\cdot}{\cdot}$ can be extended to a riemannian metric. 
\end{lemma}

\begin{proof}
Locally, we can find a frame $X_1, \dots, X_n$ for $TM$ such that $X_1, \dots X_k$ are horizontal and orthonormal with respect to $\innerproduct{\cdot}{\cdot}$. On the common domain of definition of the vectorfields $X_i$, we can define a riemannian metric $g$ via 
$$g(X_i,X_j)= \begin{cases}0 & \text{ for } i \neq j \\ 1 & \text{ for } i=j. \end{cases}$$ 
To get a riemannian metric defined on all of $M$, we patch together the obtained local riemannian metrics using a partition of unity. 
\end{proof}

\begin{lemma}\label{lem:horizontalhorizontal}
If $H$ is a regular smooth distribution, then a curve is horizontal if and only if it is absolutely continuous and its derivative lies in $H$ whenever it exists. 
\end{lemma}

\begin{proof}\mbox{}\par
\begin{enumerate}[(1)]

\item To show the one implication, let $\gamma$ be a horizontal curve. Then by definition $\gamma$ is a concatenation of curves controlled by horizontal vectorfields. Therefore it is absolutely continuous and its derivative lies in $H$ whenever it exists. 

\item It remains to show the other implication, so let $\gamma:[a,b]\rightarrow M$ be an absolutely continuous curve whose derivative is horizontal whenever it exists. Let $k$ be the rank of $H$ near $\gamma([a,b])$. We can find $a=t_0<t_1<\dots<t_n=b$ such that $\gamma([t_i,t_{i+1}])$ is contained in some open set $U_i$, $U_i$ lies in the domain of some chart, and on each set $U_i$ there is a set of orthonormal vectorfields $X_{i,1}, \dots, X_{i,k} \in \mathfrak X_H(U_i)$ spanning $H$. $\gamma$ is absolutely continuous, so by definition there are functions $v_{i,l} \in L^1([t_i,t_{i+1}];\mathbb R)$ such that for $t_i \leq t \leq t_{i+1}$ we have
$$\dot\gamma(t)=\sum_{l=1}^n v_{i,l}(t) \frac{\partial}{\partial x_l}(\gamma(t)).$$
We can extend the given subriemannian metric to a riemannian metric and use the same symbol $\innerproduct{\cdot}{\cdot}$ for it. Then for $t_i \leq t \leq t_{i+1}$ we have
\begin{equation*}\begin{split}
\dot\gamma(t) & =\sum_{j=1}^k \innerproduct{\dot\gamma(t)}{X_{i,j}(\gamma(t))}\ X_{i,j}(\gamma(t)) \\
& = \sum_{j=1}^k \underbrace{ \left( \sum_{l=1^n} v_{i,l}(t) \innerproduct{\frac{\partial}{\partial x_l}(\gamma(t))}{X_{i,j}(\gamma(t))} \right)}_{\in L^1([t_i,t_{i+1}];\mathbb R)}\ X_{i,j}(\gamma(t)),
\end{split}\end{equation*}
so on each interval $[t_i,t_{i+1}]$, $\gamma$ is controlled by horizontal vectorfields. \qedhere

\end{enumerate}
\end{proof}

\section{Basics of differential geometry}\label{sec:basics}

\subsection{Vectorfields and flows}\label{sec:vectorfieldsandflows}

The following basic definitions and Lemmas are taken from \citet{MichorNotes}.

\begin{definition}\label{def:localdiffeomorphism} \index{local diffeomorphism}
A local diffeomorphism on $M$ is a $C^\infty$-diffeomorphism from an open subset $U$ of $M$ onto an open subset $V$ of $M$. If $\varphi_i:U_i\rightarrow V_i (i=1,2)$ are local diffeomorphisms, then the composite $\varphi_1 \circ \varphi_2$ is a local diffeomorphism, with domain $\varphi_2^{-1}(U_1)$ and range $\varphi_1(V_1 \cap U_1)$. The inverse of $\varphi_1$ is denoted by $\varphi_1^{-1}$, and is a local diffeomorphism with domain $V_1$ and range $U_1$. The formal laws 
\begin{displaymath}
(\varphi_1 \circ \varphi_2) \circ \varphi_3 = \varphi_1 \circ ( \varphi_2 \circ \varphi_3) \quad \text{and}\quad 
(\varphi_1 \circ \varphi_2)^{-1} = \varphi_2^{-1} \circ \varphi_1^{-1}
\end{displaymath}
are clearly valid. 
\end{definition}

\begin{definition}\label{def:curveoflocaldiffeo} \index{curve ! of local diffeomorphisms} \index{local flow} \index{flow ! local}
A curve of local diffeomorphisms through the identity on $M$ is a smooth mapping $\varphi : U \rightarrow M$ defined on some open set $U \subset \mathbb{R} \times M$ such that
\begin{enumerate}
\item[(a)] $U$ contains $\left\{0\right\} \times M$ and for all points $x \in M$, the set $U \cap \mathbb{R} \times \{x\}$ is open and connected. 
\item[(b)] $\varphi_t:\ \ U \cap (\left\{t\right\} \times M) \rightarrow M,\ x \mapsto \varphi(t,x)$ is a local diffeomorphism.
\item[(c)] $\varphi_0 = \Id_M$.
\end{enumerate}
$\varphi$ is called a local flow if in addition, it satisfies
\begin{enumerate}
\item[(d)] $\varphi(t, \varphi(s,x)) = \varphi(t+s,x)$ for all $t, s, x$ such that both sides of the equation are defined.
\end{enumerate}
\end{definition}

\begin{definition} \index{infinitesimal generator}
The infinitesimal generator of a local flow $\varphi : U \rightarrow M$ is the vectorfield $X_\varphi$ given by $x \mapsto \evaluate{\frac{\partial}{\partial t}}{0} \varphi(t,x)$.
\end{definition}

The theory of flows and vectorfields says that for each vectorfield $X$, there is a unique local flow whose infinitesimal generator is $X$ and whose domain of definition is maximal with respect to inclusion. 

\begin{definition} \index{maximal local flow} \index{flow ! maximal local}
This flow is called the maximal local flow of $X$ and will be denoted by $Fl^X$.
\end{definition}

\begin{remark}
As in Definition \ref{def:curveoflocaldiffeo}, we will use the notations $Fl^X(t,x)$ and $Fl^X_t(x)$ to denote the same thing.
\end{remark}

\begin{remark}
If $X$ is only a local (and not a global) vectorfield, then its maximal local flow is a curve of local diffeomorphism through the identity on the domain of definition of $X$.
\end{remark}

If $f : M \rightarrow M$ is a diffeomorphism, then for any vectorfield $X \in \mathfrak{X}(M)$ the following mappings are vectorfields as well:
$$f^*X := Tf^{-1} \circ X \circ f , \quad f_*X :=Tf \circ X \circ f^{-1}.$$

\begin{definition} \index{pullback} \index{pushforward} \index{vectorfield ! pullback} \index{vectorfield ! pushforward}
In the above situation, $f^*X$ is called the pullback of $X$ along $f$, and $f_*X$ the pushforward of $X$ along $f$. 
\end{definition} 

\begin{definition} \index{f-related} \index{vectorfield ! f-related}
Let $f : M \rightarrow N$ be a smooth mapping. Two vectorsfields $X \in \mathfrak{X}(M)$ and $Y \in \mathfrak{X}(M)$ are called $f$-related, if $ Tf \circ X = Y \circ f$ holds, i.e. if the following diagram commutes:
\begin{equation*}\begin{CD}
TM @>Tf>> TN\\
@AAXA @AAYA\\
M @>f>> N
\end{CD}\end{equation*}
\end{definition} 

\begin{lemma}\label{lem:frelated}
Let $X \in \mathfrak{X}(M)$ and $Y \in \mathfrak{X}(M)$ be $f$-related vectorfields for a smooth mapping $f: M \rightarrow N$. Then we have $f \circ Fl^X_t= Fl^Y_t \circ f$, whenever both sides are defined. In particular, if $f$ is a diffeomorphism, we have $Fl^{f^*Y}_t=f^{-1} \circ Fl^Y_t \circ f$.
\end{lemma}
\begin{proof}
We have $$ \frac{\partial}{\partial t}f \circ Fl^X_t=Tf \circ \frac{\partial}{\partial t}Fl^X_t=Tf \circ X\circ Fl^X_t = Y \circ f \circ Fl^X_t $$ and  $f(Fl^X(0,x))=f(x)$. So $t \mapsto f(Fl^X(t,x))$ is an integral curve of the vectorfield $Y$ on $N$ with initial value f(x), so we have 
\begin{equation*}
f(Fl^X(t,x))=f(Fl^Y(t,f(x))) \text{ or } f \circ Fl^X_t = Fl^Y_t \circ f. \qedhere
\end{equation*}
\end{proof}

\begin{remark}
By the prescription $\mathcal L_X(f)(x):=X(x)(f)$, each vectorfield $X \in \mathfrak X(M)$ defines a derivation $\mathcal L_X$ on the algebra $C^\infty(M)$ of smooth functions on $M$, i.e. a mapping $\mathcal L_X:C^\infty(M)\rightarrow C^\infty(M)$ with $\mathcal L_X(f\cdot g)=\mathcal L_X(f)\cdot g+f\cdot \mathcal L_X(g)$. This defines a vectorspace isomorphism between $\mathfrak X(M)$ and the space of derivations on $C^\infty(M)$. Therefore we can identify a vectorfield $X$ with  the derivation $\mathcal L_X$. We will often write $X$ instead of $\mathcal L_X$ when the meaning is clear from the context. 
\end{remark}

\begin{definition} \index{Lie bracket} \index{bracket ! of vectorfields}
For $X,Y \in \mathfrak{X}(M)$, the Lie bracket $[X , Y] \in \mathfrak X(M)$ is the unique vectorfield that satisfies the formula $\mathcal L_{[X,Y]} = \mathcal L_X \circ \mathcal L_Y - \mathcal L_Y \circ \mathcal L_X$. The definition of the Lie bracket can be extended to arbitrary bracket expressions $B(X_1, \dots , X_k)$ in the obvious way. 
\end{definition}

\begin{lemma} \label{lem:frelated2}
Let $f: M \rightarrow N$ be a smooth map. If $X_1 \in \mathfrak X(M)$ is $f$-related to $Y_1 \in \mathfrak X(N)$ and $X_2 \in \mathfrak X(M)$ is $f$-related to $Y_2 \in \mathfrak X(N)$, then $\lambda_1 X_1 + \lambda_2 X_2$ is $f$-related to $\lambda_1 Y_1 + \lambda_2 Y_2$ and $[X_1,X_2]$ is $f$-related to $[Y_1,Y_2]$.
\end{lemma}

\begin{proof}
The first assertion is immediate. To prove the second, we view tangent vectors as derivations on the algebra of germs of smooth functions. Let $h$ be the germ of a smooth function on $N$ at $f(x)$. Then $h \circ f$ is the germ of a smooth function on $M$ at $x$. By assumption, we have $Tf \circ X_i = Y_i \circ f$ for $i=1,2$, thus:
\begin{multline*}
\bigl(X_i(x)\bigr)(h\circ f)=\bigl(T_x f . X_i(x)\bigr)(h)=\bigl((Tf\circ X_i)(x)\bigr)(h)=\\
=\bigl((Y_i\circ f)(x)\bigr)(h)=\bigl(Y_i(f(x))\bigr)(h).
\end{multline*}
So we have
$\mathcal L_{X_i}(h \circ f)=\bigl(\mathcal L_{Y_i}(h)\bigr) \circ f$, and therefore
\begin{align*}
\mathcal L_{[X_1,X_2]}(h \circ f) 
&=\mathcal L_{X_1} \bigl( \mathcal L_{X_2}(h \circ f)\bigr)- \mathcal L_{X_2} \bigl(\mathcal L_{X_1}(h \circ f)\bigr)\\
&= \mathcal L_{X_1}\Bigl( \bigl(  \mathcal L_{Y_2}(h) \bigr) \circ f \Bigr)-
\mathcal L_{X_2}\Bigl( \bigl(  \mathcal L_{Y_1}(h) \bigr) \circ f \Bigr)\\
&=\Bigl(\mathcal L_{Y_1} \bigl( \mathcal L_{Y_2}(h)\bigr)\Bigr)\circ f - 
\Bigl(\mathcal L_{Y_2} \bigl( \mathcal L_{Y_1}(h)\bigr)\Bigr)\circ f 
=\bigl(\mathcal L_{[Y_1,Y_2]}(h)\bigr)\circ f.
\end{align*}
But this means $Tf \circ [X_1,X_2]=[Y_1,Y_2]\circ f$.
\end{proof}

\subsection{Submanifolds}\label{sec:submanifolds}

\begin{definition} \index{submanifold}
A subset $N$ of a manifold $M$ is called a submanifold of $M$ if for each $p \in N$ there is a chart $(U, \varphi)$ in $M$ with $p \in U$ and $n \in \mathbb{N}$ such that
\begin{displaymath}
\varphi: U \rightarrow \mathbb{R}^{n+k} = \mathbb{R}^n \times \mathbb{R}^k \quad \text{and} \quad
\varphi(N \cap U) = \varphi(U) \cap \left( \mathbb{R}^n \times \left\{0\right\} \right).
\end{displaymath}
\end{definition}

One of the most important tools for dealing with submanifolds is the Constant Rank Theorem. Proofs of the Constant Rank Theorem can be found in the books of \citet{Jaenich} and \citet{Abraham}. In Section \ref{sec:thirdproofofchow}, we will give a proof of a generalization of this theorem to Banach spaces, but in finite dimensions, it takes the following form:

\begin{theorem}[Constant Rank Theorem in finite dimensions] \label{thm:finiteconstantranktheorem} 
\index{Constant Rank Theorem ! in finite dimensions}
Let $U \subset \mathbb R^m$ be an open set in $\mathbb R^m$. Let $f:U \rightarrow \mathbb R^n$ be of class $C^r$, $r \geq 1$, and assume that $f'(u)$ has constant rank $k$ in a neighbourhood of $u_0 \in \mathbb R^m$. Then there exist open sets $U_1 ,U_2 \subset \mathbb R^m$ and $V_1,V_2 \subset \mathbb R^n$ and $C^r$-diffeomorphisms $\varphi:V_1 \rightarrow V_2$ and $\psi:U_1 \rightarrow U_2$ such that $(\varphi \circ f \circ \psi)(x^1, \dots, x^m)=(x_1, \dots, x_k,0, \dots, 0)$.
\end{theorem}

\begin{definition}\index{immersion} \index{immersed submanifold}\index{submanifold ! immersed}
An immersion is a smooth map between manifolds such that its differential is injective at every point. If $i:N \rightarrow M$ is an injective immersion, then $(N,i)$ is called an immersed submanifold of $M$. 
\end{definition}

\begin{remark} \label{rem:chartsforimmersion}
An immersion has constant (full) rank at every point. So if we look at an immersion in a chart, it is a map as required in the assumptions of Theorem \ref{thm:finiteconstantranktheorem}. Therefore, in the right charts, an immersion looks like the map $(x_1, \dots, x_m) \mapsto (x_1, \dots, x_m, 0, \dots, 0)$.
\end{remark}

The following definitions and lemmas concerning initial submanifolds are taken from \citet{MichorNotes}. 

\begin{definition}
For an arbitrary subset $A$ of a manifold $M$ and $x_0 \in A$, let $C_{x_0}(A)$ denote the set of all $x \in A$ which can be connected to $x_0$ by a smooth curve in $M$ lying in $A$. 
\end{definition}

\begin{definition}\label{def:initialsubmanifold} \index{initial submanifold} \index{submanifold ! initial}
Let $N$ be a subset of a manifold $M$. A chart $(U,u)$ of $M$ centered at $x \in N$ is called an initial submanifold chart for $N$ if
$$u(C_x(U \cap N))=u(U) \cap \mathbb (R^n \times \{0\}) \text{ for some $n \in \mathbb N$.}$$

$N$ is called an initial submanifold of $M$ if for all $x \in N$, there is an initial submanifold chart centered at $x$. 
\end{definition}

\begin{definition}\label{def:initialmapping} \index{initial mapping}
We will call a mapping $i:N \rightarrow M$ initial if the following condition holds: For any manifold $Z$ a mapping $f:Z \rightarrow N$ is smooth if and only if $i \circ f : Z \rightarrow M$ is smooth. 
\end{definition}

\begin{remark}
Let $N$ be a submanifold of $M$. Then the injective immersion $i:N \rightarrow M$ is an initial mapping, and $N$ is an initial submanifold of $M$. 
\end{remark}

\begin{lemma}\label{lem:initialsubmanifold}
Let $i:N \rightarrow M$ be an injective immersion and an initial mapping. Then $i(N)$ is an initial submanifold of $M$. 
\end{lemma}

\begin{proof}
Let $n$ denote the dimension of $N$ and $m$ the dimension of $M$. Let $x \in N$. As noticed in Remark \ref{rem:chartsforimmersion}, because $i$ is an immersion, we may choose a chart $(V,v)$ centered at $i(x)$ on $M$ and another chart $(W,w)$ centered at $x$ on $N$ such that $$(v \circ i \circ w^{-1})(y^1, \dots, y^n) = (y^1, \dots, y^n, 0, \dots, 0).$$ Let $r>0$ be so small that $$\{z \in \mathbb R^m : \abs{z} < 2r \} \subset v(V) \quad \text{and}\quad \{y \in \mathbb R^n : \abs{y}<2r \} \subset w(W).$$ 
Put 
$$V_1:=v^{-1}(\{z \in \mathbb R^m : \abs{z} < 2r \}) \subset N \quad \text{and} \quad W_1 := w^{-1}(\{y \in \mathbb R^n : \abs{y}<2r \} ).$$
We claim that $(V_1,v)$ is an initial submanifold chart.
\begin{equation*}\begin{split}
v^{-1}(v & (V_1) \cap (\mathbb R^n \times \{0\})) = v^{-1}(\{(y^1, \dots, y^n, 0, \dots, 0): \abs{y}<r \})= \\
& = i \circ w^{-1} \circ (v \circ i \circ w^{-1})^{-1} (\{(y^1, \dots, y^n, 0, \dots, 0): \abs{y}<r \})= \\
& = i \circ w^{-1} (\{ y \in \mathbb R^n: \abs{y}<r \})= i(W_1) \subset C_{i(x)}(V_1 \cap i(N)).
\end{split}\end{equation*}
The last inclusion holds since $i(W_1) \subset V_1 \cap i(N)$ and $i(W_1)$ is smoothly contractible. To complete the proof, it remains to show the other inclusion $C_{i(x)}(V_1 \cap i(N)) \subset i(W_1)$. Let $z \in C_{i(x)}(V_1 \cap i(N))$. Then by definition, there is a smooth curve $c:[0,1] \rightarrow M$ with $c(0)=i(x)$, $c(1)=z$ that lies entirely in $V_1 \cap i(N)$. Because $i$ is an initial mapping, the unique curve $\tilde{c}:[0,1] \rightarrow N$ with $i \circ \tilde{c} = c$ is smooth. We claim that $\tilde{c}$ lies entirely in $W_1$. If not, then there is some $t \in [0,1]$ such that $\tilde{c}(t) \in w^{-1}(\{ y \in \mathbb R^n : r \leq \abs{y} < 2r \})$ since $\tilde c$ is smooth and thus continuous. But then we have 
\begin{multline*}
(v \circ c)(t)=(v \circ i)(\tilde c(t)) \in (v \circ i \circ w^{-1})(\{ y \in \mathbb R^n : r \leq \abs{y} < 2r \}) = \\
= \{ (y,0) \in (\mathbb R^n \times \{0\}) : r \leq \abs{y} < 2r \} \subset \{ z \in \mathbb R^m : r \leq \abs{z} < 2r \}.
\end{multline*}
But $(v \circ c)(t) \in \{ z \in \mathbb R^m : r \leq \abs{z} < 2r \}$ means $c(t) \notin V_1$, a contradiction. So $\tilde c$ lies entirely in $W_1$, thus $z = c(1) = i(\tilde c(1)) \in i(W_1)$. 
\end{proof}

\begin{lemma}\label{lem:initialsubmanifolduniquestructure}
Let $N$ be an initial submanifold of $M$. Then there is a unique smooth manifold structure on $N$ such that the injection $i:N \rightarrow M$ is an injective immersion and an initial mapping. 
\end{lemma}

\begin{proof}
Uniqueness follows directly from the definition of an initial mapping. To prove existence, we claim that the charts $(C_x(U_x \cap N),u_x)$, where $(U_x, u_x)$ is an initial submanifold chart for $N$, yield a smooth manifold structure for $N$ such that $i$ has the desired properties. 

The sets $C_x(U_x \cap N)$ might not be open in the subspace topology of $N$, so in general, the subspace topology of $N$ is not the same as the chart topology. But we know that the chart topology is finer than the subspace topology, and therefore it is Hausdorff. The chart changings are smooth since they are just restrictions of chart changings on $M$. Clearly, $i:N \rightarrow M$ is an injective immersion. 

We need to show that $i$ is an initial mapping. For $z \in Z$, we choose a chart $(U,u)$ centered at $f(z)$, such that $u(C_{f(z)}(U \cap N)) = u(U) \cap (\mathbb R^n \times \{0\})$. Then $f^{-1}(U)$ is open in $Z$ and contains a chart $(V,v)$ centered at $z$ with $v(V)$ a ball. Then $f(V)$ is smoothly contractible in $U \cap N$, so $f(V) \subset C_{f(z)}(U \cap N)$, and $(u|C_{f(z)}(U \cap N))\circ f \circ v^{-1}$ is smooth. 
\end{proof}

\begin{remark}
Let $N$ be an initial submanifold of $M$. Then the connected components of $N$ are separable because $M$ admits a riemannian metric which can be induced on $N$ (see \citet[Section~5.5]{Abraham}). However, $N$ may have uncountably many connected components.
\end{remark}

\section{Chow's Theorem}\label{sec:chow}

\begin{definition}\label{def:involutive} \index{involutive}
A set of local vectorfields $\mathcal V \subset \mathfrak X_{loc}(M)$ is called involutive if for any $X,Y \in \mathcal V$, also $[X,Y] \in \mathcal V$. We will call a distribution $H$ involutive if the set of local horizontal vectorfields $\mathfrak X_{loc,H}$ is involutive. 
\end{definition}

\begin{definition}\label{def:liehull}\label{def:bracketgenerating}\index{Lie hull}\index{bracket generating}\index{H\"ormanders condition}
For a set of local vectorfields $\mathcal V \subset \mathfrak X_{loc}(M)$, the Lie hull $\mathcal{L}(\mathcal V)$ of $\mathcal V$ is the smallest involutive set of local vectorfields which contains $\mathcal V$. $\mathcal V$ is called bracket generating if $\mathcal L(\mathcal V)$ spans $TM$. A smooth distribution $H$ is called bracket generating (satisfies H\"ormander's condition) if the set of locally defined horizontal vectorfields $\mathfrak X_{loc,H}(M)$ is bracket generating. 
\end{definition}

\begin{theorem}[Chow's Theorem]\label{thm:chow}\index{Chow's Theorem}
If $H$ is a bracket generating distribution on a connected manifold, then any two points in the manifold can be connected by a horizontal path. 
\end{theorem}

The converse of Chow's Theorem fails: It is possible that any two points can be connected by a horizontal path, but that still the distribution is not bracket generating. In Section \ref{sec:firstproofofchow} we will see that this is the case exactly if $\mathcal S(\mathfrak X_{loc,H}(M))$ spans $TM$, but $\mathcal L(\mathfrak X_{loc,H}(M))$ does not. 

\begin{example}
The following very easy example for this situation can be found in the paper of \citet{Nagel} and in the book of \citet[page~65]{Jurdjevic}. For an other, more elaborate example see \citet[page~24]{Montgomery}. Let $H$ be the distribution on $\mathbb R^2$ that is spanned by the two vectorfields 
$$X_1 = \frac{\partial}{\partial x}, \quad X_2=\phi \frac{\partial}{\partial y},$$ 
where $\phi$ is a $C^\infty$-function on $\mathbb R^2$ that vanishes on the left-hand plane $x \leq 0$ and is positive otherwise. Then any two points in the plane can be connected by a horizontal curve. But $H$ is not bracket generating because for $x < 0$, $H_x$ is one-dimensional and therefore all the Lie brackets vanish. 
\end{example}

Figure \ref{fig:irregularballs} shows that in the above example, one encounters irregularities in the shape of the subriemannian ball $B(\epsilon,x)$ as $\epsilon$ increases. The Ball-Box Theorem \ref{ballboxtheorem} and also Theorem \ref{thm:endpointmapisopen} show that these irregularities are ruled out in the case of a bracket generating distribution. So non-bracket generating distributions are not desirable from a geometric point of view.   

\begin{empfile}[graphics/irregularballs]
\begin{empdef}[irregularballs](3cm,3cm)
u=h/2;
picture temp;

pickup pencircle scaled 1bp;
draw (-u,0)--(-u/3,0);
pickup defaultpen;
draw (0,-u)--(0,u);
dotlabel.bot(btex $x$ etex , (-2/3u,0));
temp:=currentpicture shifted (-3u,0);
currentpicture:=nullpicture;

pickup pencircle scaled 1bp;
draw (-4/3u,0)--(0,0);
pickup defaultpen;
draw (0,-u)--(0,u);
dotlabel.bot(btex $x$ etex , (-2/3u,0));
addto temp also currentpicture;
currentpicture:=nullpicture;

pickup pencircle scaled 1bp;
path p; 
p=(halfcircle scaled 2/3u rotated -90)--(-1/3u,0)--cycle; 
fill p withcolor 0.7white;
draw p;
draw (-4/3u,0)--(-1/3u,0);
pickup defaultpen;
draw (0,-u)--(0,-1/3u);
draw (0,u)--(0,1/3u);
dotlabel.bot(btex $x$ etex , (-2/3u,0));
addto temp also currentpicture shifted (3u,0);
currentpicture:=temp;
\end{empdef}
\end{empfile}
\immediate\write18{cd graphics ; mpost -tex=latex irregularballs ; cd .. }

\begin{figure}[htb]
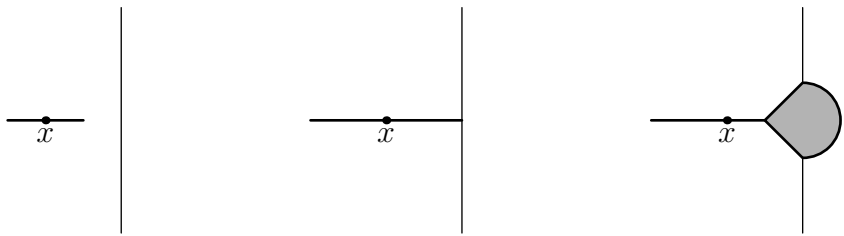
 
\begin{center}
\empuse{irregularballs}
\caption{$B(\epsilon,x)$ for $x=(-1,0) \in \mathbb R^2$ and $\epsilon=\frac{1}{2}, \epsilon=1, \epsilon=\frac{3}{2}$.}\label{fig:irregularballs}
\end{center}
\end{figure}
 
We will give three different proofs of Chow's Theorem in Sections \ref{sec:firstproofofchow}, \ref{sec:secondproofofchow} and \ref{sec:thirdproofofchow}, respectively.

%% file: chapters/firstproofofchow.tex
%
%

\subsection{First proof of Chow's theorem} \label{sec:firstproofofchow}

This proof works under the most general assumptions compared to the proofs given in Section \ref{sec:secondproofofchow} and \ref{sec:thirdproofofchow}. It is based on the theory of \citet{Stefan} and \citet{Sussmann}. However, the notation and the proofs are taken from the lecture notes of \citet{MichorNotes}. The outline of the proof is the following: \index{Stefan-Sussman Theory}The Stefan-Sussmann-theory asserts that the accessible set is a maximal integral manifold of a certain integrable distribution containing $H$. This will be shown in Theorem \ref{thm:accessiblesetisintegralmanifold}. The bracket generating condition implies that this distribution is equal to $TM$. Because $M$ is connected, it is the maximal integral manifold of this distribution. Using Theorem \ref{thm:accessiblesetisintegralmanifold}, the proof of Chow's theorem is only a few lines long; it can be found at the very end of this section. 

\subsubsection{Integral manifolds} \label{sec:integralmanifolds}

\begin{definition}\label{integralmanifold} \index{integral manifold} \index{maximal integral manifold} \index{distribution ! integrable} \index{integrable distribution}
An integral manifold of a distribution $H$ is a connected immersed submanifold $(N, i)$ such that $T_x i(T_x N)=H_{i(x)}$ for all $x \in N$. An integral manifold of $H$ is called maximal, if it is not contained in any strictly larger integral manifold of $H$. A distribution $H$ is called integrable if for any $x \in M$, there is an integral manifold of $H$ containing $x$. 
\end{definition}

\begin{remark}\label{rem:distributionwithfullrank}
Let $H$ be a distribution on a connected manifold $M$. Then $M$ is a maximal integral manifold of $H$ if and only if $H=TM$.
\end{remark}

\begin{lemma}\label{lem:integralmanifold}
Let $H$ be a smooth distribution on $M$. Then we have:
\begin{enumerate}[(1)]
\item \label{lem:integralmanifold1}
If $(N,i)$ is an integral manifold of $H$ and $X \in \mathfrak X_{loc,H}(M)$, then $i^*X$ makes sense and is an element of $\mathfrak X_{loc}(N)$, which is $i | i^{-1}(U_X)$-related to $X$, where $U_X \subset M$ is the open domain of $X$.

\item \label{lem:integralmanifold2}
If $(N_1,i_1)$ and $(N_2,i_2)$ are integral manifolds of $H$, then 
$i_1^{-1}(i_1(N_1) \cap i_2(N_2))$ and $i_2^{-1}(i_1(N_1) \cap i_2(N_2))$
are open subsets in $N_1$ and $N_2$, respectively; furthermore $i_2^{-1} \circ i_1$ is a diffeomorphism between them.

\item \label{lem:integralmanifold3}
If $x \in M$ is contained in some integral manifold of $H$, then it is contained in a unique maximal one. 
\end{enumerate}
\end{lemma}

\begin{proof}\indent\par
\begin{itemize}
\item[(\ref{lem:integralmanifold1})] If $i(x) \in U_X$ for $x \in N$, we have $X(i(x)) \in H_{i(x)}=T_x i(T_x N)$, so $i^*X(x) :=((T_x i)^{-1} \circ X \circ i)(x)$ makes sense. It is defined on an open subset of $N$ and is smooth in $x$.

\item[(\ref{lem:integralmanifold2})] Let $X \in \mathfrak X_{loc, H}(M)$. Then $i^*_j X \in \mathfrak X_{loc}(N_j)$ and is $i_j$-related to $X$ for $j=1,2$. By Lemma \ref{lem:frelated} we have $$i_j \circ Fl_t^{i^*_j X}=Fl_t^X \circ i_j.$$ Now choose $x_j \in N_j$ such that $i_1(x_1)=i_2(x_2)=x \in M$ and choose vectorfields $X_1, \dots, X_k \in \mathfrak X_{loc,H}$ such that $\left\{ X_1(x), \dots, X_k(x) \right\}$ form a basis of $H_x$. Then $$f_j(t_1, \dots, t_k) :=(Fl_{t_1}^{i^*_j X_1} \circ \dots \circ Fl_{t_k}^{i^*_j X_k})(x_j)$$ is a smooth mapping defined near zero $\mathbb{R}^k \rightarrow N_j$. Since $\evaluate{\frac{\partial}{\partial t_i}}{0} f_j = i^*_j X_i$ for $j=1,2$, we see that $f_j$ is a diffeomorphism near zero. Finally we have
\begin{equation*}\begin{split}
(i_2^{-1} \circ i_1 \circ f_1)(t_1, \dots, t_k) 
& = (i_2^{-1} \circ i_1 \circ Fl_{t_1}^{i^*_1 X_1} \circ \dots \circ Fl_{t_k}^{i^*_1 X_k})(x_1)\\
& = (i_2^{-1} \circ Fl_{t_1}^{X_1} \circ \dots \circ Fl_{t_k}^{X_k} \circ i_1)(x_1) \\
& = (Fl_{t_1}^{i^*_2 X_1} \circ \dots \circ Fl_{t_k}^{i^*_2 X_k} \circ i_2^{-1} \circ i_1)(x_1)\\
& = f_2(t_1, \dots, t_k).
\end{split}\end{equation*}
So $i_2^{-1} \circ i_1$ is a diffeomorphism, as required. 

\item[(\ref{lem:integralmanifold3})] Let $N$ be the union of all integral manifolds containing $x$. Choose the union of all the atlases of these integral manifolds as atlas for $N$, which is a smooth atlas by (2). Note $M$ admits a riemannian metric which can be induced on $N$, so $N$ is separable (see \citet[section~5.5]{Abraham}). \qedhere
\end{itemize}
\end{proof}

\begin{remark}\index{leaf of a foliation} \index{foliation}
If $H$ is an integrable distribution on a manifold $M$, then by Lemma \ref{lem:integralmanifold} each point is contained in a unique maximal integral manifold. These maximal integral manifolds form a partition of $M$. This partition is called the foliation of $M$ induced by the integrable distribution $H$, and each maximal integral manifold is called a leaf of the foliation. 
\end{remark}

\begin{definition} \index{stable}
A set of local vectorfields $\mathcal V \subset \mathfrak X_{loc}(M)$ is called stable if for all $X, Y \in \mathcal V$ and for all $t$ for which it is defined, the local vectorfield $(Fl_t^X)^* Y$ is again an element of $\mathcal V$. Let $\mathcal S(\mathcal V)$ denote the minimal stable set of local vectorfields containing $\mathcal V$. 
\end{definition}

\begin{lemma}\label{lem:elementsofsv}
Let $\mathcal V \subset \mathfrak X_{loc}(M)$ be a set of local vectorfields. Then $\mathcal S(\mathcal V)$ consists of all local vectorfields of the form $(Fl_{t_1}^{X_1} \circ \dots \circ Fl_{t_k}^{X_k})^* Y$ for $X_i, Y \in \mathcal V$. 
\end{lemma}

\begin{proof}
We have to show that the set of vectorfields of the above form is stable. By Lemma \ref{lem:frelated}, the flow of such a vectorfield is 
\begin{equation*}
Fl^{(Fl_{t_1}^{X_1} \circ \dots \circ Fl_{t_k}^{X_k})^* Y}_t=Fl_{-t_k}^{X_k} \circ \dots Fl_{-t_1}^{X_1} \circ Fl_t^Y \circ Fl_{t_1}^{X_1} \circ \dots \circ Fl_{t_k}^{X_k}. \qedhere
\end{equation*}
\end{proof}

\begin{definition}\index{infinitesimal automorphism}
A local vectorfield $X \in \mathfrak X_{loc}(M)$ is called an infinitesimal automorphism of a distribution $H$ if $T_x(Fl_t^X)(H_x) \subset H_{Fl^X(t,x)}$ whenever defined. Let $\aut(H)$ denote the set of all infinitesimal automorphisms of $H$. 
\end{definition}

\begin{lemma}
$\aut(H)$ is stable. 
\end{lemma}

\begin{proof}
Let $X,Y \in \aut(H)$, and let $Z:=(Fl_{t_1}^X)^* Y$. Then 
\begin{equation*}
T_x(Fl_t^Z)(H_x) = T_x(Fl_{-t_1}^X \circ Fl_t^Y \circ Fl_{t_1}^X)(H_x) \subset H_{Fl^Z(t,x)}. \qedhere
\end{equation*}
\end{proof}

\begin{theorem}[Stefan-Sussmann] \label{thm:sussmann}
Let $H$ be a smooth distribution on a manifold $M$. Then the following conditions are equivalent:
\begin{enumerate}[(1)]
\item \label{thm:sussmann1} $H$ is integrable. 
\item \label{thm:sussmann2} $\mathfrak X_{loc,H}(M)$ is stable. 
\item \label{thm:sussmann3} There exists a subset $\mathcal V \subset \mathfrak X_{loc}(M)$ such that $\mathcal S(\mathcal V)$ spans $H$.
\item \label{thm:sussmann4} $\aut(H) \cap \mathfrak X_{loc,H}(M)$ spans $H$.
\end{enumerate}
\end{theorem}

\begin{proof}\mbox{}\par
\begin{description}
\item[](\ref{thm:sussmann1}) $\Rightarrow$ (\ref{thm:sussmann2}) Let $X \in \mathfrak X_{loc,H}(M)$ and let $L$ be the leaf through $x \in M$, with $i:L \rightarrow M$ the inclusion. Then $Fl_{-t}^X \circ i = i \circ Fl_{-t}^{i^* X}$ by Lemma \ref{lem:frelated}, so we have
\begin{equation*}\begin{split}
T_x(Fl_{-t}^X)(H_x) & = T(Fl_{-t}^X) . T_x i . T_x L = T(Fl_{-t}^X \circ i) . T_x L \\
& = T i . T_x(Fl_{-t}^{i^* X}) . T_x L \\
& = T i . T_{Fl^{i^* X}(-t,x)}L = H_{Fl^X(-t,x)}.
\end{split}\end{equation*}
This implies that $(Fl_t^X)^* Y \in \mathfrak X_{loc,H}(M)$ for any $Y \in \mathfrak X_{loc,H}(M)$.

\item[](\ref{thm:sussmann2}) $\Rightarrow$ (\ref{thm:sussmann4}) In fact, (\ref{thm:sussmann2}) says that $\mathfrak X_{loc,H} \subset \aut(H)$.

\item[](\ref{thm:sussmann4}) $\Rightarrow$ (\ref{thm:sussmann3}) We can choose $\mathcal V = \aut(H) \cap \mathfrak X_{loc,H}(M)$: for $X,Y \in \mathcal V$, we have $(Fl_t^X)^* Y \in \mathfrak X_{loc,H}(M)$; so $\mathcal V \subset \mathcal S(\mathcal V) \subset \mathfrak X_{loc,H}(M)$ and $H$ is spanned by $\mathcal V$.

\item[](\ref{thm:sussmann3}) $\Rightarrow$ (\ref{thm:sussmann1}) We have to show that each point $x \in M$ is contained in some integral submanifold for the distribution $H$. Let $\dim H_x = n$. There are $X_1, \dots, X_n \in \mathcal S(\mathcal V)$ such that $X_1(x), \dots, X_n(x)$ span $H_x$. As in the proof of Lemma \ref{lem:integralmanifold} (2), we consider the mapping 
$$f(t_1, \dots, t_n) := (Fl_{t_1}^{X_1} \circ \dots \circ Fl_{t_n}^{X_n})(x)$$ 
defined near zero in $\mathbb R^n$. Since the rank of $f$ at zero is $n$, the image under $f$ of a small open neighbourhood of zero is a submanifold $N$ of $M$. We claim that $N$ is an integral manifold of $H$. The tangent space $T_{f(t_1, \dots, t_n)}N$ is linearly generated by the the tangent vectors 
\begin{equation*}\begin{split}
&\frac{\partial}{\partial t_k}  (Fl_{t_1}^{X_1} \circ \dots \circ Fl_{t_n}^{X_n})(x)  =\\ 
& \quad = T(Fl_{t_1}^{X_1} \circ \dots \circ Fl_{t_{k-1}}^{X_{k-1}}) X_k((Fl_{t_{k+1}}^{X_{k+1}} \circ \dots \circ Fl_{t_n}^{X_n})(x)) \\
& \quad = ((Fl_{-t_1}^{X_1})^* \dots (Fl_{-t_{k-1}}^{X_{k-1}})^* X_k)(f(t_1, \dots, t_n)).
\end{split}\end{equation*}
Since $\mathcal S(\mathcal V)$ is stable, these vectors lie in $H_{f(t_1, \dots, t_n)}$. They even span $H_{f(t_1, \dots, t_n)}$ because $\dim H_{f(t)} = \dim H_x = n$, as we will prove now. Indeed, since $\mathcal S(\mathcal V)$ spans $H$ and is stable, $T_x(Fl_t^X)$ is a vector space isomorphism from $H_x$ onto $H_{Fl^X(t,x)}$ for all $X \in \mathcal S(\mathcal V)$, and the conclusion follows from the form of $f$. \qedhere
\end{description}
\end{proof}

\begin{theorem}\label{thm:globalsussmann}
Let $H$ be an integrable distribution of a manifold $M$. Then for each $x \in M$ there exists a chart $(U,u)$ around $x$ with $u(U)=\{y \in \mathbb R^m : \abs{y^i}<\epsilon \text{ for all } i \}$ for some $\epsilon >0$, and a countable subset $A \subset \mathbb R^{m-n}$, such that for the leaf $L$ through $x$ we have $$u(U \cap L) = \{ y \in u(U) : (y^{n+1}, \dots, y^m) \in A \}.$$ Each leaf is an initial submanifold. If furthermore the distribution has locally constant rank, the theorem holds with the same $n$ for each leaf $L$. 
\end{theorem}

\begin{definition} \label{def:distinguishedchart} \index{distinguished chart} \index{plaque}
The chart $(u,U)$ of the above theorem is called a distinguished chart for the distribution or the foliation. A connected component of $U \cap L$ is called a plaque. 
\end{definition}

\begin{proof}[Proof of Theorem \ref{thm:globalsussmann}]
Let $L$ be the leaf through $x$, $\dim L = n$. Let $X_1, \dots, X_n \in \mathfrak X_{loc,H}(M)$ be local vectorfields such that $X_1(x), \dots, X_n(x)$ is a basis of $H_x$. We choose a chart $(V,v)$ centered at $x \in M$ such that the vectors 
$$X_1(x), \dots, X_n(x), \evaluate{\frac{\partial}{\partial v^{n+1}}}{x}, \dots, \evaluate{\frac{\partial}{\partial v^m}}{x}$$
form a basis of $T_x M$. Then 
$$f(t_1, \dots, t_n) := (Fl_{t_1}^{X_1} \circ \dots \circ Fl_{t_n}^{X_n}) (v^{-1}(0, \dots, 0, t^{n+1}, \dots, t^m)$$
is a diffeomorphism from a neighbourhood of zero in $\mathbb R^n$ onto a neighbourhood of $x$ in $M$. Let $(U,u)$ be the chart given by $f^{-1}$, suitably restricted. We have 
$$y \in L \Longleftrightarrow (Fl_{t_1}^{X_1} \circ \dots \circ Fl_{t_n}^{X_n})(y) \in L$$
for all $t_1, \dots, t_n$ and all $y$ for which both expressions make sense. So we have 
$$f(t_1, \dots, t_n) \in L \Longleftrightarrow f(0, \dots, 0, t^{n+1}, \dots, t^m) \in L,$$
and consequently $L \cap U$ is the disjoint union of connected sets of the form $\{ y \in U : (u^{n+1}(y), \dots, u^m(y)) = \text{ constant } \}$. Since $L$ is a connected immersed submanifold of $M$, it is second countable and only a countable set of constants can appear in the description of $U \cap L$ given above. We see that $u(C_x(L \cap U)) = u(U) \cap (\mathbb R^n \times \{0\})$, so we have found initial submanifold charts for $L$, and therefore $L$ is an initial submanifold. According to the proof of Lemma \ref{lem:initialsubmanifolduniquestructure}, the charts that make $(L,i)$ an immersed submanifold with the property that $i:L \rightarrow M$ is an initial mapping are given by constraining $u$ to $C_x(L \cap U)$. However, in Lemma \ref{lem:integralmanifold}, $L$ as an integral manifold has already been given a smooth manifold structure. We will see immediately that the two structures are the same: $u | C_x(L \cap U)$ is the inverse of the map $(t_1, \dots t_n) \mapsto ( Fl_{t_1}^{X_1} \circ \dots \circ Fl_{t_n}^{X_n})(x)$, and therefore $u | C_x(L \cap U)$ is just one of the charts given to $L$ in Lemma \ref{lem:integralmanifold}. 

The arguments given above are valid for any leaf of dimension $n$ meeting $U$, so also the assertion for an integrable distribution of locally constant rank follows. 
\end{proof}

\subsubsection{Orbits and integral manifolds} \label{sec:orbits}

\begin{definition} \index{pseudo-group of local diffeomorphisms}
A pseudo-group of local diffeormorphisms is a set $G$ of local diffeormorphisms such that for each $\varphi, \psi \in G$, we have $$\varphi \circ \psi \in G \text{ and } \varphi^{-1} \in G,$$ 
where $\varphi \circ \psi$ and $\varphi^{-1}$ are as in Definition \ref{def:localdiffeomorphism}. For $\mathcal V \subset \mathfrak X_{loc,H}(M)$, let $G_{\mathcal V}$ denote the smallest pseudo-group of local diffeomorphisms which contains the flows of all vectorfields in $\mathcal V$.
\end{definition}

\begin{remark}\label{rem:elementsofgv}
The elements of $G_{\mathcal V}$ are precisely the maps of the form $$Fl^{X_1}_{t_1} \circ \dots \circ Fl^{X_n}_{t_n}, \text{ with } X_i \in \mathcal V.$$ 
\end{remark}

\begin{definition}\index{G-equivalent @ $G$-equivalent} \index{orbit}
Let $G$ be a pseudo-group of local diffeomorphisms on $M$. Two points $x,y \in M$ are called $G$-equivalent if there is a $\varphi \in G$ such that $\varphi(x) = y$. This defines an equivalence relation on $M$. Its equivalence classes are called the orbits of $G$ or $G$-orbits. 
\end{definition}

\begin{theorem}\label{thm:orbitisintegralmanifold}
Let $\mathcal V \subset \mathfrak X_{loc}(M)$ and $x \in M$. Then the $G_{\mathcal V}$-orbit through $x$ is the maximal integral manifold through $x$ of the distribution spanned by $\mathcal S(\mathcal V)$. 
\end{theorem}

\begin{proof}
The distribution spanned by $\mathcal S(\mathcal V)$ is integrable because of the Stefan-Sussmann Theorem \ref{thm:sussmann}. Let $L$ be the leaf through $x$ of the distribution spanned by $\mathcal S(\mathcal V)$, and let $O$ be the $G_{\mathcal V}$-orbit containing $x$. 

\begin{enumerate}[(1)]

\item To show $L \subset O$, we will prove that any two points in $L$ are $G_{\mathcal V}$-equivalent. By the part (\ref{thm:sussmann3})\ $\Rightarrow$\ (\ref{thm:sussmann1}) of the proof of the Stefan-Sussmann Theorem \ref{thm:sussmann} we know that the map $$f(t_1, \dots, t_n) := (Fl_{t_1}^{X_1} \circ \dots \circ Fl_{t_n}^{X_n})(y),$$ where $X_i \in \mathcal S(\mathcal V)$ and $y \in L$, is a diffeomorphism from an open connected neighbourhood of zero in $\mathbb R^n$ onto an open subset $U$ of $L$. Because of Lemma \ref{lem:elementsofsv} and Lemma \ref{lem:frelated}, $Fl_{t_i}^{X_i}$ is the concatenation of flows of vectorfields in $\mathcal V$. Therefore by definition any point in $U$ is $G_{\mathcal V}$-equivalent to $y$. Because $y$ was chosen arbitrarily, every point $y \in L$ has a neighbourhood whose points are all $G_{\mathcal V}$-equivalent. So the $G_{\mathcal V}$-equivalence classes in $L$ are open in $L$ and form a partition of $L$. Because $L$ is connected, all but one of these equivalence classes are empty, and therefore all points in $L$ are $G_{\mathcal V}$-equivalent. 

\item It remains to show the other inclusion $O \subset L$. Let $y \in O$, i.e. 
$$y=Fl^{X_1}_{t_1} \circ \dots \circ Fl^{X_k}_{t_k}(x) \quad \text{with}\quad X_i \in \mathcal V \subset \mathcal S(\mathcal V).$$ Then $y \in L$ because $L$ is stable under the flows of vectorfields in $\mathcal S(\mathcal V)$. This follows from Lemma \ref{lem:frelated} applied to the pullback $i^* X$ of a vectorfield, where $i:L \rightarrow M$ is the inclusion. \qedhere

\end{enumerate}
\end{proof}

\subsubsection{Accessible set and integral manifolds} \label{sec:accessibleset}

\begin{theorem} \label{thm:accessiblesetisintegralmanifold}
Let $H$ be a regular distribution on $M$ and $x \in M$. Then the accessible set $Acc(x)$ is the leaf through $x$ of the distribution spanned by $\mathcal S(\mathfrak X_{loc,H}(M))$. 
\end{theorem}

\begin{proof}
Let $E$ be the distribution spanned by $\mathcal S(\mathfrak X_{loc,H}(M))$. Then $E$ is integrable because the Stefan-Sussmann Theorem \ref{thm:sussmann}. Let $L$ be the leaf through $x$ of the distribution $E$.

\begin{enumerate}[(1)]

\item To show $L \subset Acc(x)$, we have to show that $L$ is horizontally pathwise connected. This is a consequence of Theorem \ref{thm:orbitisintegralmanifold} with $\mathcal V=\mathfrak X_{loc,H}(M)$ because any $\mathfrak X_{loc,H}(M)$-orbit is of course horizontally connected. 

\item To show the other inclusion $Acc(x) \subset L$, we have to show that a horizontal path starting in $L$ remains in $L$. In fact, it suffices to show that a path controlled by horizontal vectorfields remains in $L$. So let $\gamma:[0,1]\rightarrow M$ be an absolutely continuous path such that 
$$\dot\gamma(t)=\sum_{j=1}^k u_j(t)\ X_j(\gamma(t))$$
with $L^1$-functions $u_j$. We can pull the vectorfields $X_j$ back to $L$ along the inclusion $i:L \rightarrow M$ and look at the differential equation
$$\dot\gamma(t)=\sum_{j=1}^k u_j(t)\ (i^*X_j)(\gamma(t)), \quad \gamma(0)=i^{-1}(x)$$
on $L$. By restricting the controls to some smaller interval $[0,\epsilon]$, we can make their $L^1$-norm so small that we can apply Theorem \ref{thm:controlledcurve}, which gives us a solution $\tilde\gamma:[0,\epsilon]\rightarrow L$. By uniqueness, $\gamma$ and $\tilde\gamma$ coincide on $[0,\epsilon]$. From these arguments we see that $\gamma$ locally remains in $L$. Therefore the set $\{t \in [0,1]: \gamma(t) \in L \}$ is open and non-empty. But also the set $\{t \in [0,1]: \gamma(t) \notin L \}$ is open for the same reasons. Because $[0,1]$ is connected, one of these sets must be empty, and therefore $\gamma$ remains in $L$. \qedhere
\end{enumerate}
\end{proof}

\subsubsection{Proof of Chow's Theorem}

\begin{lemma} \label{lem:integrabledistributionisinvolutive}
An integrable distribution is involutive. 
\end{lemma}

\begin{proof}
Let $H$ be an integrable distribution; we have to show that $\mathfrak X_{loc,H}$ is involutive. Take vectorfields $X,Y \in \mathfrak X_{loc,H}$ defined near $x \in M$, and let $(N,i)$ be an integral manifold of $H$ near $x$. By Lemma \ref{lem:integralmanifold}.\ref{lem:integralmanifold1}, $i^*X$ and $i^*Y$ are local vectorfields on $N$, and they are $i$-related to $X$ and $Y$, respectively. By Lemma \ref{lem:frelated2}, also $[i^*X,i^*Y]$ is $i$-related to $[X,Y]$. Therefore $[X,Y]$ is horizontal. 
\end{proof}

\begin{theorem}[Chow's Theorem]\label{thm:chow1}\index{Chow's Theorem}
If $H$ is a bracket generating distribution on a connected manifold, then any two points in the manifold can be connected by a horizontal path. 
\end{theorem}

\begin{proof}
Let $E$ denote the integrable distribution spanned by $\mathcal S(\mathfrak X_{loc,H}(M))$, and let $x \in M$. By the bracket generating condition and Lemma \ref{lem:integrabledistributionisinvolutive}, $E=TM$. $M$ is connected, so $M$ is the leaf through $x$ of $E$. By Theorem \ref{thm:accessiblesetisintegralmanifold} $Acc(x)$ also is the leaf through $x$ of $E$, so $Acc(x)=M$. 
\end{proof}

%% file: chapters/secondproofofchow.tex
%
%

\subsection{Second proof of Chow's Theorem}\label{sec:secondproofofchow}

This proof works only for regular distributions. The idea of the proof comes from \citet{Montgomery}. He derives Chow's theorem from the Ball-Box Theorem, which is interesting in itself. 

\begin{theorem}[Ball-box Theorem]\label{ballboxtheorem}\index{Ball-box Theorem}
Let $H$ be a regular bracket generating distribution on an $n$-dimensional manifold $M$, and let $x_0 \in M$. Then there exist coordinates $y_1, \dots , y_n$ centered at $x_0$ and positive constants $c < C$ and $\epsilon_0 > 0$
 such that for all $\epsilon < \epsilon_0$, 
 \begin{displaymath}
 \MyBox^w(c \epsilon) \subset B(\epsilon, x_0) \subset \MyBox^w(C \epsilon).
 \end{displaymath}
Here, $w$ is  some vector in $\mathbb{N}^n$ called the weighting associated to the growth vector at $x_0$ and $\MyBox^w(\epsilon)$ is the $w$-weighted  box of size $\epsilon$ given in coordinates by the expression
\index{weighted box} \index{box}
\begin{displaymath}
\MyBox^w(\epsilon) = \left\{ y \in \mathbb{R} : \abs{y_i} \leq \epsilon^{w_i}, i = 1, \dots , n \right\}.
\end{displaymath}
\end{theorem}
 
To prove Chow's Theorem it is sufficient to show that the subriemannian ball $B(\epsilon, x_0)$ is a neighbourhood of $x_0$, and this is all we will prove here. The reason is that we were confronted with technical difficulties in the proof given in the book of \citet{Montgomery}. We refer to \citet{Nagel} for the full proof and a much longer, but very nice and complete treatment of the subject. 

The outline of the proof presented in this work is the following: Locally, some horizontal vectorfields $X_1, \dots, X_k$ together with a selection of their brackets span the tangent space. To these vectorfields correspond the flows $Fl^{X_1}_t, \dots, Fl^{X_k}_t$ and a selection of brackets of these flows. It will take some time to calculate the first non-vanishing derivative at $t=0$ of the brackets of flows, but this is done in Theorem \ref{michorstheorem}: The first non-vanishing derivative at $t=0$ of a bracket of $Fl^{X_1}_t, \dots, Fl^{X_k}_t$ is the same bracket applied to $X_1, \dots, X_k$. In Theorem \ref{schwachesballboxtheorem} this result will allow us to construct a $C^1$-diffeomorphism $\psi$ from the concatenation of the above flows and their brackets. Finally the openness of $B(\epsilon,x_0)$ follows from the openness of $\psi$. 

\begin{definition}
For sets of local vectorfields $\mathcal V, \mathcal W \subset \mathfrak X_{loc}(M)$, we define 
\begin{equation*}
\begin{split}
\left[\mathcal V, \mathcal W \right] &= \left\{ \left[ X , Y \right] : X \in\mathcal V, Y \in \mathcal W \right\} \quad \text{and}\\
\mathcal V + \mathcal W &= \left\{ X + Y : X \in \mathcal V, Y \in \mathcal W \right\}.  
\end{split}
\end{equation*}
For a set of local vectorfields $\mathcal V$, we inductively define the subsets 
\begin{equation*}
\mathcal V_1 = \mathcal V, \quad \mathcal V_{k+1} = \mathcal V_k + [\mathcal V,\mathcal V_k].
\end{equation*}
\end{definition}

\begin{definition}
Let $H$ be a smooth distribution. Then the distribution spanned by $\mathfrak X_{loc,H}(M)_k$ will be denoted by $H_k$. 
\end{definition}

\begin{definition}
Let $C^\infty(M)$ denote the set of smooth functions from $M$ to $\mathbb{R}$. For a set of locally defined vectorfields $\mathcal{V}$, let 
\begin{displaymath}
\linearspan_{C^\infty(M)} \mathcal{V} = 
\left\{ \sum u_i X_i : u_i \in C^\infty(M), X_i \in \mathcal V  \right\},
\end{displaymath}
where the number of summands in the above sum has to be finite. 
\end{definition}

\begin{lemma}\label{lem:spanxhochk}
Let $\mathcal{V}$ be a set of locally defined vectorfields. Then 
\begin{equation*}
(\linearspan_{C^\infty(M)} \mathcal{V})_k = \linearspan_{C^\infty(M)} (\mathcal{V}_k).
\end{equation*}
Particularly, if $\mathcal V$ spans $H$, then $\mathcal V_k$ spans $H_k$. 
\end{lemma}

\begin{proof}
Let $\mathcal W=\linearspan_{C^\infty(M)} (\mathcal{V})$.

\begin{enumerate}[(1)]
\item $\linearspan_{C^\infty(M)} (\mathcal{V}_k) \subset \mathcal W_k$: 
Clearly, $\mathcal{V} \subset \mathcal W$, so we have $\mathcal V_k \subset \mathcal W_k$, and 
$$\linearspan_{C^\infty(M)} (\mathcal{V}_k) \subset \linearspan_{C^\infty(M)} (\mathcal W_k) =\mathcal W_k.$$ The last equality holds because $\mathcal W_k$ is a $C^\infty(M)$-module, as will be shown right now. Indeed we claim that if $\mathcal W$ is a $C^\infty(M)$-module, then $\mathcal W_k$ is a $C^\infty(M)$-module, too. This is trivial for $k=1$. For $k \geq 1$ and $X \in \mathcal W_{k+1}$, we can write 
$$X=Y+[Z,W] \in \mathcal W_k + [\mathcal W, \mathcal W_k].$$ 
Let $f \in C^\infty(M)$. Then 
$$f X= f Y + (W f) Z + [f Z,W] \in \mathcal W_k +  [\mathcal W, \mathcal W_k]=\mathcal W_{k+1},$$
and this proves our claim. 
 
\item $\mathcal W_k \subset \linearspan_{C^\infty(M)} (\mathcal{V}_k)$: This can be proven by induction. For $k = 1$, the statement is trivial. For the inductive step, let $X  \in \mathcal W_{k+1} = \mathcal W_k + \left[ \mathcal W, \mathcal W_k \right]$. By the inductive assumption, $\mathcal W_k \subset \linearspan_{C^\infty(M)} (\mathcal{V}_k)$. Therefore, with $X_i \in \mathcal{V}_k$, $Y_i \in \mathcal{V}, Z_j \in \mathcal{V}_k$, we can write
\begin{equation*}\begin{split}
X 	&= \sum u_i X_i + \left[ \sum v_i Y_i , \sum w_j Z_j \right] \\
	&= \sum u_i X_i + \sum \sum v_i w_j \left[ Y_i, Z_j \right] + v_i (Y_i w_j) Z_j - w_j (Z_j v_i) Y_i \\
	 & \in  \linearspan_{C^\infty(M)} (\mathcal{V}_{k+1}). \qedhere
\end{split}\end{equation*}
\end{enumerate}
\end{proof}

\begin{lemma} \label{lem:bracketgeneratinglocalr}
If $H$ is a bracket generating distribution, then locally, there is an $r \in \mathbb{N}$ such that $H_r=TM$.
\end{lemma}

\begin{proof}
Let $n$ be the dimension of $M$. For $x \in M$, we find vectorfields $X_1, \dots, X_n \in \mathcal{L}(\mathfrak X_{loc,H}(M))$ defined near $x$ such that $X_1(x), \dots, X_n(x)$ form a basis of $T_{x}M$. Because $X_1, \dots, X_n$ are linearly independent at $x$, they are also linearly independent in some neighbourhood of $x$. Furthermore, because 
$$\mathcal{L}(\mathfrak X_{loc,H}(M)) = \sum_{k \in \mathbb{N}} \mathfrak X_{loc,H}(M)_k,$$ every $X_i$ lies in some $\mathfrak X_{loc,H}(M)_{k_i}$. Therefore, they all lie in $\mathfrak X_{loc,H}(M)_r$, where $r=\max \{ k_i : i = 1, \dots , n \}$. Consequently, $H_r=TM$ near $x$. 
\end{proof}

\begin{definition} \index{growth vector}
For a bracket generating distribution such that $H_r=TM$, we define the growth vector $n(x) \in \mathbb{R}^r$ at $x \in M$ to be the vector whose components $n_i(x)$ are the rank of the distribution $H_i$ at $x$. (Obviously $n(x)$ may vary from point to point.)
\end{definition}

The following theorem and its proof is taken from a paper of \citet{Michor}; it describes the correspondence between brackets of flows and brackets of vectorfields.

\begin{definition} \index{bracket ! of local diffeomorphisms}
For curves of local diffeomorphisms $\varphi$ and $\psi$, we define the bracket of local diffeomorphisms
\begin{displaymath}
\left[ \varphi, \psi \right](t,x) = \left( \psi_t^{-1} \circ \varphi_t^{-1} \circ \psi_t \circ \phi_t \right) (x)
\end{displaymath}
for all $(t,x)$ where the equation is well-defined. As for Lie brackets of vectorfields, we extend this definition in the obvious way to arbitrary formal bracket expressions. 
\end{definition}

\begin{theorem}\label{michorstheorem}
For $i = 1 \dots k$, let $\varphi^i$ be curves of local diffeomorphisms through the identity. Let $X_i$ be their infinitesimal generators. Then for each formal bracket expression $B$ of length $k$ we have 
\begin{eqnarray*}
0 & = & \evaluate{\frac{\partial^l}{\partial t^l}}{0} B(\varphi^1_t, \dots , \varphi^k_t)\quad \text{for}\ 1 \leq l < k, \\
B(X_1, \dots , X_k) & = & \frac{1}{k!} \evaluate{\frac{\partial^k}{\partial t^k}}{0} B(\varphi^1_t, \dots , \varphi^k_t).
\end{eqnarray*}
\end{theorem}
In the above theorem, $B(X_1, \dots , X_k)$ clearly is a vectorfield, but what about the $k$-th derivative of $B(\varphi^1_t, \dots , \varphi^k_t)$? To explain this, we need the following lemma: 

\begin{lemma}\label{lem:michor1}
Let $c : \mathbb{R} \rightarrow M$ be a smooth curve. If $c(0) = x \in M$, $c'(0) = 0, \dots, c^{(k-1)}(0)=0$, then $c^{(k)}(0)$ is a well-defined tangent vector in $T_{x}M$ which is given by the derivation $f \mapsto (f \circ c)^{(k)}(0)$. 
\end{lemma}

\begin{proof}
We have to show that the Leibnitz rule holds for the map $f \mapsto (f \circ c)^{(k)}(0)$.
\begin{eqnarray*}
\left( \left(f \cdot g \right) \circ c \right)^{(k)} (0) & = &
\left( \left(f \circ c \right) \cdot \left( g \circ c \right) \right) ^ {(k)} =\\ 
& = & \sum_{j=0}^{k} \binom{k}{j} \left( f \circ c \right)^{(j)}(0) \cdot \left( g \circ c \right)^{(k-j)}(0) = \\
& = & \left( f \circ c\right)^{(k)}(0) \ g(q) + f(q) \ \left( g \circ c\right)^{(k)} (0)
\end{eqnarray*}
since all other summands vanish: $\left( f \circ c \right)^{(j)}(0) = 0$ for $j = 1, \dots , k-1$. 
\end{proof}

Lemmas \ref{lem:michor2}, \ref{lem:michor3} and \ref{lem:michor4} constitute the proof of Theorem \ref{michorstheorem}. In the following, we will use the notation $\partial^k_t = \frac{\partial^k}{\partial t^k}$.

\begin{lemma}\label{lem:michor2}
Let $\varphi, \psi$ be curves of local diffeomorphisms through $\Id_M$ and let $f \in C^{\infty}(M)$. Then we have 
$$\evaluate{\partial^k_t}{0}(\varphi_t\circ \psi_t)^*f=\sum_{j=0}^{k} \binom{k}{j}(\evaluate{\partial^j_t}{0}\psi_t^*)(\evaluate{\partial^{k-j}_t}{0} \varphi_t^*)f.$$
For $l$ curves of local diffeormphisms $\varphi^1,\dots \varphi^l$ through $\Id_M$ the multinomial version of this formula also holds:
$$\evaluate{\partial^k_t}{0}(\varphi^1_t\circ\dots\circ \varphi^l_t)^*f=\sum_{j_1+\dots +j_l=k}\frac{k!}{j_1!\dots j_l!}(\evaluate{\partial^{j_l}_t}{0}(\varphi^l_t)^*)\dots(\evaluate{\partial^{j_1}_t}{0}(\varphi^1_t)^*)f.$$
\end{lemma}

\begin{proof}
We will only show the binomial version. For a function $h(t,s)$ of two variables we have
\[\partial^k_t h(t,t)=\sum_{j=0}^{k} \evaluate{\binom{k}{j}\partial^j_t\partial^{k-j}_s h(t,s)}{t=s}\]
since for $h(t,s)=f(t)g(s)$ this is just a consequence of the Leibnitz rule, and linear combinations of such decomposable tensors are dense in the space of all functions of two variables in the compact $C^\infty$-topology, so that by continuity the formula holds for all functions. In the following form it implies the lemma:
\begin{equation*} 
\evaluate{\partial^k_t}{0}f(\varphi(t,\psi(t,x)))=\sum_{j=0}^{k}\binom{k}{j}\evaluate{\partial^{k-j}_s \partial^j_t f(\varphi(t,\psi(s,x)))}{t=s=0}. \qedhere
\end{equation*}
\end{proof}

\begin{lemma}\label{lem:michor3}
Let $\varphi$ be a curve of local diffeormphisms through $\Id_M$ with first non-vanishing derivative 
$$ k!\ X= \evaluate{\partial^k_t}{0}\varphi_t.$$
Then the inverse curve of local diffeomorphims $\varphi^{-1}$ has first non-vanishing derivative 
$$-k!\ X= \evaluate{\partial^k_t}{0}\varphi^{-1}_t.$$
\end{lemma}

\begin{proof}
We have $\varphi^{-1}_t \circ \varphi_t=\Id_M$, so by Lemma \ref{lem:michor2} we get for $1\leq j\leq k$
\begin{eqnarray*}
0&=&\evaluate{\partial^j_t}{0}(\varphi^{-1}_t \circ \varphi_t)^*f = \sum_{i=0}^j\binom{i}{j} (\evaluate{\partial^i_t}{0}(\varphi_t)^*)(\evaluate{\partial^{j-i}_t}{0}(\varphi^{-1}_t)^*)f \\
&=&\evaluate{\partial^j_t}{0}(\varphi_t)^*(\varphi^{-1}_0)^*f+(\varphi_0)^* \evaluate{\partial^j_t}{0}(\varphi^{-1}_t)^*f \\
&=&\evaluate{\partial^j_t}{0}(\varphi_t)^*f+\evaluate{\partial^j_t}{0}(\varphi^{-1}_t)^*f
\end{eqnarray*}
as required.
\end{proof}

\begin{lemma}\label{lem:michor4}
Let $\varphi$ be a curve of local diffeormphisms through $\Id_M$ with first non-vanishing derivative $m!\ X= \evaluate{\partial^m_t}{0}\varphi_t$, and let $\psi$ be a curve of local diffeormphisms through $\Id_M$ with first non-vanishing derivative $n!\ Y= \evaluate{\partial^n_t}{0}\psi_t$.
Then the curve of local diffeomorphisms $[\varphi,\psi]$ has first non-vanishing derivative
$$(m+n)!\ [X,Y]=\evaluate{\partial^{m+n}_t}{0}[\varphi,\psi]_t.$$
\end{lemma}

\begin{proof}
By the multinomial version of Lemma \ref{lem:michor2} we have
\begin{eqnarray*}
A_Nf &:=& \evaluate{\partial^N_t}{0}(\psi^{-1}_t\circ\varphi^{-1}_t\circ \psi_t\circ\varphi_t)^*f \\
&=& \sum_{i+j+k+l=N} \frac{N!}{i!j!k!l!} (\evaluate{\partial^i_t}{0}(\varphi^{-1}_t)^*) (\evaluate{\partial^j_t}{0}(\psi^{-1}_t)^*) (\evaluate{\partial^k_t}{0}(\varphi_t)^*) (\evaluate{\partial^l_t}{0}(\psi_t)^*)f.
\end{eqnarray*}
Let us suppose that $1\leq n\leq m$, the case $m\leq n$ is similar. If $N<n$ all summands vanish. If $N=n$ we have by Lemma \ref{lem:michor3}
\begin{eqnarray*}
A_Nf= (\evaluate{\partial^n_t}{0}(\varphi_t)^*)f+(\evaluate{\partial^n_t}{0}(\psi_t)^*)f+(\evaluate{\partial^k_t}{0}(\varphi^{-1}_t)^*)f+(\evaluate{\partial^k_t}{0}(\psi^{-1}_t)^*)f=0.
\end{eqnarray*}
If $n<N\leq m$ we have, using again Lemma \ref{lem:michor3}
\begin{eqnarray*}
A_Nf &=& \sum_{j+l=N}\frac{N!}{j!l!}(\evaluate{\partial^j_t}{0}(\psi_t)^*)(\evaluate{\partial^k_t}{0}(\psi^{-1}_t)^*)f + \delta^m_N(\evaluate{\partial^m_t}{0}(\varphi_t)^*)f+\evaluate{\partial^m_t}{0}(\varphi^{-1}_t)^*)f)\\
&=& (\evaluate{\partial^N_t}{0}(\psi^{-1}_t\circ\psi_t)^*)f+0=0.
\end{eqnarray*}
Now we come to the difficult case $m,n<N\leq m+n.$ By Lemma \ref{lem:michor2} and since all other terms vanish we get
\begin{eqnarray*}
A_Nf&=&\evaluate{\partial^N_t}{0}(\psi^{-1}_t\circ \varphi^{-1}_t\circ \psi_t)^*f \\
&+& \binom{N}{m}(\evaluate{\partial^m_t}{0}\varphi^*_t)(\evaluate{\partial^{N-m}_t}{0}(\psi^{-1}_t\circ \varphi^{-1}_t\circ \psi_t)^*)f+(\evaluate{\partial^N_t}{0}\varphi^*_t)f.
\end{eqnarray*}
By Lemma \ref{lem:michor2} again we get
\begin{equation*}\begin{split}
&\evaluate{\partial^N_t}{0}(\psi^{-1}_t\circ \varphi^{-1}_t\circ \psi_t)^*f=\\
& \quad = \sum_{j+k+l=N}\frac{N!}{j!k!l!}(\evaluate{\partial^j_t}{0}(\psi_t)^*)(\evaluate{\partial^k_t}{0}(\varphi^{-1}_t)^*)(\evaluate{\partial^l_t}{0}(\psi^{-1}_t)^*)f \\
& \quad =\sum_{j+l=N}\binom{N}{j}(\evaluate{\partial^j_t}{0}(\psi_t)^*)(\evaluate{\partial^l_t}{0}(\psi^{-1}_t)^*)f+\binom{N}{m}(\evaluate{\partial^{N-m}_t}{0}(\psi_t)^*)(\evaluate{\partial^m_t}{0}(\varphi^{-1}_t)^*)f+\\
&\quad\quad\quad+\binom{N}{m}(\evaluate{\partial^{m}_t}{0}(\varphi^{-1}_t)^*)(\evaluate{\partial^m_t}{0}(\psi^{-1}_t)^*)f + \evaluate{\partial^N_t}{0}(\varphi^{-1}_t)^*)f\\
& \quad =\binom{N}{m}(\evaluate{\partial^{N-m}_t}{0}(\psi_t)^*)m!\ (-X f)+\binom{N}{m}(\evaluate{\partial^{m}_t}{0}(\varphi^{-1}_t)^*)n!\ (-Y f) +\evaluate{\partial^N_t}{0}(\varphi^{-1}_t)^*f\\
& \quad =\delta^N_{m+n}(m+n)!\ (XY-YX)f+\evaluate{\partial^N_t}{0}(\varphi^{-1}_t)^*f\\
& \quad =\delta^N_{m+n}(m+n)!\ [X,Y]f+\evaluate{\partial^N_t}{0}(\varphi^{-1}_t)^*f.
\end{split}\end{equation*}
From the second expression above one can also read off that
\begin{equation*}
\evaluate{\partial^{N-m}_t}{0}(\psi^{-1}_t\circ \varphi^{-1}_t\circ \psi_t)^*f= \evaluate{\partial^{N-m}_t}{0}((\varphi^{-1}_t)^*f.
\end{equation*}
If we put the last conclusions together we get, using Lemma \ref{lem:michor2} and Lemma \ref{lem:michor3} again:
\begin{equation*}\begin{split}
A_N f &= \delta^N_{m+n}(m+n)!\ [X,Y]f+\evaluate{\partial^N_t}{0}(\varphi^{-1}_t)^*f\\
&\quad+\binom{N}{m}(\evaluate{\partial^m_t}{0}\varphi_t^*)(\evaluate{\partial^{N-m}_t}{0}(\varphi^{-1}_t)^*)f+(\evaluate{\partial^N_t}{0}\varphi_t^*)f\\
&=\delta^N_{m+n}(m+n)!\ [X,Y]f+(\evaluate{\partial^N_t}{0}(\varphi_t^{-1}\circ \varphi_t)^*f\\
&=\delta^N_{m+n}(m+n)!\ [X,Y]f. \qedhere
\end{split}\end{equation*}
\end{proof}

We will use the following two lemmas in our proof of Chow's Theorem.

\begin{lemma}  \label{ersteswurzellemma}
Let $\varphi: (-\epsilon, \epsilon) \rightarrow M$ be a smooth curve whose derivatives at zero of order up to and including $n-1$ are zero, i.e.
$\varphi'(0) = 0, \dots, \varphi^{(n-1)}(0) = 0$. Then $t \mapsto \varphi(\sqrt[n]{t}\,)$ is a $C^1$-function on $[0,\epsilon)$ whose first derivative at zero equals $\frac{1}{n!} \varphi^{(n)}(0)$. If $n$ is odd, it is a $C^1$-function on $(-\epsilon, \epsilon)$.
\end{lemma}

\begin{proof}
This is a local question, so we can work in a chart and assume without loss of generality that $M = \mathbb{R}^n$ and $\varphi(0)=0$. First, we will look at $t \geq 0$ only and show differentiability from the right at $t=0$ using the Taylor formula $\varphi(t) = \frac{t^n}{n!} \varphi^{(n)}(0)+o(t^n)$.
\begin{displaymath}
\lim_{t \searrow 0} \frac{1}{t} \left( \varphi(\sqrt[n]{t}\,) - \varphi(0) \right) =
\lim_{t \searrow 0} \frac{1}{t} \left( \frac{t}{n!} \varphi^{(n)}(0) + o(t) \right) =
\frac{\varphi^{(n)}(0)}{n!}.
\end{displaymath}
Now, we show that the derivative is continuous from the right at $t=0$. We use the Taylor formula for $\varphi'$ and get
$\varphi'(t)=\varphi^{(n)}(0) \frac{t^{n-1}}{(n-1)!}+o(t^{n-1})$, and therefore
\begin{eqnarray*}
\lim_{t \searrow 0} \frac{d}{dt} \varphi(\sqrt[n]{t}\,) & = &
\lim_{t \searrow 0} \varphi'(\sqrt[n]{t}\,) \frac{t^{(1/n)-1}}{n} = \\
& = & \lim_{t \searrow 0} \varphi^{(n)}(0) \frac{\sqrt[n]{t}^{n-1}}{(n-1)!} \frac{t^{(1/n)-1}}{n}
+ o(\sqrt[n]{t}^{n-1}) \frac{t^{(1/n)-1}}{n} =
\frac{\varphi^{(n)}(0)}{n!}.
\end{eqnarray*}
If $n$ is odd, all the above calculations can be done for $t<0$ as well. 
\end{proof}

\begin{lemma} \label{zweiteswurzellemma}
Let $\varphi_1, \varphi_2 : (-\epsilon, \epsilon) \rightarrow M$ be smooth functions whose derivatives at zero of order up to and including $n-1$ are zero and $\varphi_1^{(n)}(0) = -\varphi_2^{(n)}(0)$. Then the function 
\begin{displaymath}
\psi: (-\epsilon, \epsilon) \rightarrow M, \quad
\psi(t) = 
\begin{cases}
\varphi_1(\sqrt[n]{t}\,) &		\text{for $t \geq 0$}, \\
\varphi_2(-\sqrt[n]{-t}\,) &	\text{for $t < 0$} 
\end{cases}
\end{displaymath} 
is a $C^1$ function.
\end{lemma}

\begin{proof}
Again, this is a local question, so we can work in a chart and assume without loss of generality that $M = \mathbb{R}^n$ and $\varphi_1(0)=\varphi_2(0)=0$. We have to show that the left and right derivatives of $\psi$ at $t=0$ are equal. According to Lemma \ref{ersteswurzellemma}, the right derivative of $\psi$ at $0$ is $\frac{1}{n!} \varphi_1^{(n)}(0)$. 
To calculate the left derivative, we look at the function $t \mapsto \varphi_2(-t)$. Its derivatives at zero of order up to and including $n-1$ are zero, and its $n$-th derivative at zero is $(-1)^n \varphi_2^{(n)}(0)$. So we can apply Lemma \ref{ersteswurzellemma} to it and calculate
\begin{eqnarray*}
\lim_{t \nearrow 0} \frac{1}{t} \left( \psi(t) - \psi(0) \right) & = &
\lim_{t \nearrow 0} \frac{1}{t} \varphi_2(-\sqrt[n]{-t}) = 
-\lim_{t \searrow 0} \frac{1}{t} \varphi_2(-\sqrt[n]{t}) = \\
& = & - \frac{1}{n!} (-1)^n \varphi_2^{(n)}(0) = \frac{1}{n!} (-1)^n \varphi_1^{(n)}(0) = \frac{1}{n!} \varphi_1^{(n)}(0).
\end{eqnarray*}
\end{proof}

\begin{theorem}\label{schwachesballboxtheorem}
If $H$ is a regular bracket generating distribution on a manifold $M$, then for all $\epsilon > 0$, the subriemannian ball $B(\epsilon, x_0)$ is a neighbourhood of $x_0$.  
\end{theorem}

\begin{proof}
\begin{enumerate}[(1)]

\item Let $n$ be the dimension of $M$. We may work locally, so Lemma \ref{lem:bracketgeneratinglocalr} asserts that there exists an $r \in \mathbb{N}$ such that $H_r=TM$. Let $(n_1, \dots , n_r)$ be the growth vector at $x_0$. Because $H$ is regular, we can take orthonormal horizontal vectorfields $X_1, \dots, X_{n_1}$ defined near $x_0$ spanning $H$ near $x_0$. $\mathcal V = \{X_1,\dots, X_{n_1}\}$ spans $H$, so Lemma \ref{lem:spanxhochk} asserts that $\mathcal V_j$ spans $H_j$ for $j=1, \dots, r$. So there are bracket expressions $B_1, \dots, B_n$ such that the vectorfields $X_i = B_i(X_1, \dots , X_{n_1})$ span $TM$, and such that for $1 \leq i \leq n_1$, $B_i$ is simply a vectorfield, for $n_1<i \leq n_2$, $B_i$ consists of one bracket etc. For $i=1, \dots, n$, let $w_i$ denote the number of brackets in $B_i$ increased by one. $w_i$ is called the weighting of the vectorfield $X_i$. 

\item Let $\varphi_1, \dots, \varphi_{n_1}$ be the local flows of the vectorfields $X_1, \dots, X_{n_1}$. So $\varphi_1, \dots, \varphi_{n_1}$ are curves of local diffeormorphisms through the identity. Now, for $i=1, \dots, n$, let $\varphi_i = B_i(\varphi_1, \dots, \varphi_{n_1})$. Note that this does not change the definition of $\varphi_1, \dots, \varphi_{n_1}$ because $B_1, \dots, B_{n_1}$ are bracket expressions of length zero. According to Theorem \ref{michorstheorem}, in any coordinate system centered at $x$ we have 
$\varphi_i(t,x) = t^{w_i} X_i(x) + o(t^{w_i})$. If $w_i$ is even, $t^{w_i}$ is never negative, so we can not move in the negative $X_i$ direction. This motivates the approach of adapting the commutators of flows. 

\item For $i=1, \dots, n$, we build the formal bracket expression $C_i$ by swapping the arguments of the outermost bracket in $B_i$. So, if $E_i$ and $F_i$ are bracket expressions such that $B_i = \left[ E_i, F_i \right]$, then $C_i = \left[ F_i, E_i \right]$. Then we define
\begin{displaymath}
\chi_i(t,x)=
\begin{cases}
B_i(\varphi_1, \dots, \varphi_{n_1})(t,x) &		\text{for $w_i$ even, $t \geq 0$, } \\
C_i(\varphi_1, \dots, \varphi_{n_1})(t,x) &		\text{for $w_i$ even, $t \geq 0$, } \\
B_i(\varphi_1, \dots, \varphi_{n_1})(t,x) & 	\text{for $w_i$ odd. }
\end{cases}
\end{displaymath}

\item For $i=1,\dots , n$, we write 
\begin{displaymath}
\sigma_i : \mathbb{R} \rightarrow \mathbb{R}, \quad \sigma_i(t) := 
\begin{cases} 
t^{w_i}  & 		\text{for $w_i$ even, $t \geq 0$, }\\
-t^{w_i} &		\text{for $w_i$ even, $t < 0$, }\\
t^{w_i}  &		\text{for $w_i$ odd.}
\end{cases}
\end{displaymath}
Then $\sigma_i$ is a non-smooth homemorphism. Let $s_i$ be its inverse, i.e. 
\begin{displaymath}
s_i : \mathbb{R} \rightarrow \mathbb{R}, \quad s_i(t) = 
\begin{cases} 
\sqrt[w_i]{t}   & 		\text{for $w_i$ even, $t \geq 0$, }\\
-\sqrt[w_i]{-t} &		\text{for $w_i$ even, $t < 0$, }\\
\sqrt[w_i]{t}   &		\text{for $w_i$ odd.}
\end{cases}
\end{displaymath}
Let $\psi_i(t,x) := \chi_i(s_i(t),x)$ and let 
\begin{eqnarray*}
\chi(t_1, \dots, t_n) & := & \chi_1(t_1) \circ \dots \circ \chi_n(t_n)(x_0) \\
\psi(t_1, \dots, t_n) & := & \psi_n(t_n) \circ \dots \circ \psi_1(t_1)(x_0). 
\end{eqnarray*}
According to Lemmas \ref{ersteswurzellemma} and \ref{zweiteswurzellemma}, the maps $\psi_1, \dots, \psi_n$ are $C^1$ in $t$ for each fixed $x$. For each fixed $t$, they are $C^\infty$ in $x$. Therefore $\psi$ is a $C^1$-map defined on a neighbourhood of $0 \in \mathbb{R}^n$. We calculate its partial derivatives at $0$: 
\begin{displaymath}
\evaluate{\frac{\partial}{\partial t_i}}{0} \psi = \frac{X_i(x_0)}{w_i!}
\quad \text{because} \quad
\evaluate{\frac{\partial}{\partial x}}{(0,0)} \psi_i = \Id. 
\end{displaymath}
$X_1(x_0), \dots, X_n(x_0)$ are linearly independent tangent vectors, so we may apply the inverse function theorem and invert $\psi$ in a neighbourhood of $0$. The inverse of $\psi$ is $C^1$, so it is continuous, and therefore $\psi$ is an open mapping. 

\item The map $\chi$ is a composition of $\chi_1(t_1), \dots, \chi_n(t_n)$, and each map $\chi_i$ is a bracket of the flows $\varphi_1, \dots, \varphi_{n_1}$ which correspond to the horizontal vectorfields $X_1, \dots, X_{n_1}$. If we expand the map $\chi$ in terms of the flows $\varphi_1, \dots, \varphi_{n_1}$, let $K$ denote the number of flows that appear in this expansion. Now, we claim that $\chi(\MyBox(\epsilon / K))$ is contained in the subriemannian ball $B(\epsilon, x_0)$, where 
$$\MyBox(\epsilon) = \left\{ y \in \mathbb{R}^n : \abs{y_i} \leq \epsilon, i = 1, \dots, n \right\}.$$

To prove the claim, we note that for $1 \leq i \leq n_1$ and any point $x$ in the domain of definition of $\varphi_i(t_i)$, the subriemannian distance $d(x,\varphi_i(t_i)(x))$ is less than $\epsilon / K$ if $\abs{t_i} < \epsilon / K$. The reason is that $x$ and $\varphi_i(t_i)(x)$ can be connected by the horizontal path $t \mapsto \varphi_i(t, x), 0 \leq t \leq t_i$. The length of this path is $\abs{t_i}$ because $\varphi_i$ is the flow of the vectorfield $X_i$, and $X_1, \dots, X_{n_1}$ were taken to be orthonormal. Now the claim follows because $x_0$ and $\chi(t_1, \dots, t_n)$ can be connected by a concatenation of $K$ horizontal paths, each of length less than $\epsilon / K$. 

\item Define the weighted box $Box^w(\epsilon) = \left\{ y \in \mathbb{R} : \abs{y_i} \leq \epsilon^{w_i}, i = 1, \dots , n \right\}$ as in the Ball-Box Theorem \ref{ballboxtheorem}. $\MyBox^w(\epsilon / K)$ is a neighbourhood of $0 \in \mathbb{R}^n$ and $\psi$ is an open mapping, therefore $\psi(\MyBox^w(\epsilon / K))$ is a neighbourhood of $x_0$. Because $B(\epsilon, x_0) \supset \chi(\MyBox(\epsilon / K) = \psi(\MyBox^w(\epsilon / K))$, we have shown that $B(\epsilon, x_0)$ is a neighbourhood of $x_0$ as well. \qedhere
\end{enumerate}
\end{proof}

\begin{lemma}\label{lem:locallyhorizontallypathwise}
If $M$ is connected and every point has a horizontally pathwise connected neighbourhood, then $M$ is horizontally pathwise connected.
\end{lemma}

\begin{proof}
Any connected and locally pathwise connected topological space is pathwise connected. Here, we only allow horizontal paths, but the proof remains the same: Take an arbitrary point $x_0 \in M$. Then $Acc(x_0)$ is open. Let $Acc(x_0)^C$ denote the set of all points not lying in $Acc(x_0)$. $Acc(x_0)^C$ is open because it is the union
$\bigcup \left\{ Acc(x) : x \notin Acc(x_0) \right\}$. Because $M$ is connected, either $Acc(x_0)$ or $Acc(x_0)^C$ is empty. Because $x_0 \in Acc(x_0)$, $Acc(x_0)^C$ must be empty, and $Acc(x_0) = M$.
\end{proof}

\begin{theorem}[Chow's Theorem]\label{thm:chow2} \index{Chow's Theorem}
If $H$ is a regular bracket generating distribution on a connected manifold, then any two points in the manifold can be connected by a horizontal path. 
\end{theorem}

\begin{proof}
Theorem \ref{schwachesballboxtheorem} asserts that every point $x$ in the manifold has a horizontally pathwise connected neighbourhood, namely $B(\epsilon,x)$, so we can apply Lemma \ref{lem:locallyhorizontallypathwise}. 
\end{proof}

%% file: chapters/thirdproofofchow.tex
%
%

\subsection{Third proof of Chow's Theorem}\label{sec:thirdproofofchow}

For this proof we need the assumption that $H$ is locally finitely generated.  The outline of the proof is the following: By Lemma \ref{lem:locallyhorizontallypathwise} it is sufficient to prove that every point has a horizontally pathwise connected neighbourhood, so we can work locally and assume that $X_1, \dots, X_k$ are defined everywhere and that $M$ is an open connected subset of $\mathbb R^n$. We will see in Theorem \ref{thm:controlledcurve} that under suitable conditions, the differential equation 
\begin{equation*} 
\dot \gamma(t)=\sum_{i=1}^k u_i(t) X_i(\gamma(t)), \quad \gamma(0)=x
\end{equation*}
has a unique solution $\gamma$, and $\gamma$ depends smoothly on $u$ and $x$. Then also the endpoint map, that is the map assigning to $u$ the endpoint of $\gamma$ is smooth. Using some calculus in Banach spaces we will show in Theorem \ref{thm:endpointmapinitialsubmanifold} that the image of the endpoint map is an initial submanifold of $M$. If $H$ is bracket generating, this initial submanifold has the same dimension as $M$, and therefore it is an open subset of $M$. In fact, it is all of $M$, and this will conlude the proof of Chow's Theorem. 

\begin{remark}
The most important part in the proof is the fact that the endpoint map is a smooth map from a Banach space to $M$. We can ask if the set of horizontal curves can be endowed with the structure of a Banach manifold such that the endpoint map is a smooth map from this Banach manifold to $M$. The idea is to expand the derivative of $\gamma$ in terms of a local horizontal frame to get coordinates for $\gamma$. For regular distributions, this is done in \citet[Appendix~D]{Montgomery}.
\end{remark}

\subsubsection{Calculus in Banach spaces}

A good treatment of calculus in Banach spaces, including the following theorems and definitions, can be found in the book of \citet{Abraham}. 

\begin{definition} \index{multilinear operator} \index{bounded linear operator} \index{linear operator}
If $X_1, \dots, X_k$ and $Y$ are Banach spaces, let $L(X_1,\dots, X_k;Y)$ denote the Banach space of $k$-multilinear bounded operators of $X_1, \dots, X_k$ to $Y$. If $X_i=X, 1 \leq i \leq k$, this space is denoted by $L^k(X,Y)$. 
\end{definition}

\begin{remark} 
There are (natural) norm-preserving isomorphisms
\begin{equation*}\begin{split}
L(X_1,L(X_2, \dots, X_k)) & \cong L(X_1, \dots, X_k;Y) \cong L(X_1, \dots, X_{k-1};L(X_k;Y)) \\
& \cong L(X_{i_1},\dots,X_{i_k};Y)
\end{split}\end{equation*}
where $i_1, \dots, i_k$ is a permutation of $(1,\dots,k)$.
\end{remark}

\begin{definition} \label{def:frechetderivative} \index{Fr\'echet derivative}
Let $X$ and $Y$ be Banach spaces, and $U$ an open subset of $X$. A map $f:U \rightarrow Y$ is called (Fr\'echet-)differentiable at $x \in U$ if there exists a linear bounded map $L \in L(X,Y)$ such that $$\lim_{h \rightarrow 0} \frac{\norm{f(x+h)-f(x)-L h}_Y}{\norm{h}_X}=0.$$
If it exists, a map $L$ with the above property is unique. $L$ will be denoted $f'(x)$, $\frac{\ud f}{\ud x}(x)$ or $\frac{\partial f}{\partial x}(x)$. Moreover, if $f$ is differentiable at every point in $U$ and if the map $$f':U \rightarrow L(X,Y), \quad x \mapsto f'(x)$$ is continuous, we say $f$ is of class $C^1$ or continuously differentiable. Proceeding inductively, we define $$f^{(k)}=(f^{(k-1)})':U \rightarrow L^k(X,Y)$$ if it exists. Here we have identified $L(X,L^{k-1}(X,Y))$ with $L^k(X,Y)$. If $f^{(k)}$ exists and is continuous, we say $f$ is of class $C^k$. $f$ is said to be of class $C^\infty$ or smooth if it is of class $C^k$ for all $k \geq 1$. 
\end{definition}

\begin{definition} \index{contraction}
Let $U$ be a subset of a Banach space. Then a map $F: U \rightarrow U$ is called a contraction if there is $L<1$ such that for all $x, \widetilde x \in U$
$$\norm{F(x)-F(\widetilde x)} \leq L \norm{x-\widetilde x}.$$
\end{definition}

\begin{theorem}[Banach Fixed Point Theorem]\label{thm:banach1} \index{Banach Fixed Point Theorem}
Let $U$ be a closed subset of a Banach space $X$ and $F:U \rightarrow U$ a contraction. Then $F$ has a unique fixed point.
\end{theorem}

\begin{proof}
Choose $x_0 \in U$. We define a sequence $x_n$ inductively by $x_{n+1}:=F(x_n)$. First we will show by induction that for all  $n\in \mathbb{N}$ 
$$\norm{x_{n+1}-x_n} \leq L^n \norm{x_1-x_0}.$$ 
For $n=0$ the statement is trivial. For the inductive step we get
\begin{eqnarray*}
\norm{x_{n+2}-x_{n+1}}=\norm{F(x_{n+1})-F(x_{n})}\leq L\norm{x_{n+1}-x_n} \leq L^{n+1} \norm{x_1-x_0}. 
\end{eqnarray*}
We use this result to show that $x_n$ is a Cauchy sequence. Because $0\leq L<1$ we can find for all $\epsilon >0$ an $N\in \mathbb{N}$ with 
$$L^N < \epsilon \frac{(1-L)}{\norm{x_1-x_0}}.$$
Using the triangle inequality we get for all $m,n \in \mathbb{N}$ with $m>n\geq N$
\begin{eqnarray*}
\norm{x_m-x_n} &\leq& \norm{x_{m}-x_{m-1}}+\norm{x_{m-1}-x_{m-2}}+ \cdots +\norm{x_{n+1}-x_n} \\
&\leq&L^{m-1}\norm{x_1-x_0}+L^{m-2}\norm{x_1-x_0}+ \cdots +L^n \norm{x_1-x_0} \\
&=&\norm{x_1-x_0} L^n\sum_{k=0}^{m-n-1}L^k < \norm{x_1-x_0} L^n\sum_{k=0}^{\infty}L^k \\
&=&\norm{x_1-x_0} L^n \frac{1}{1-L}<\epsilon.
\end{eqnarray*}
Now $x_n$ is a Cauchy sequence and converges to a point $x^* \in U$. $x^*$ is a fixed point of $F$ because 
$$\norm{x^*-F(x^*)}=\lim_{n\rightarrow \infty} \norm{x_n-F(x_n)}=\lim_{n\rightarrow \infty} \norm{x_n-x_{n+1}}=0.$$
To prove uniqueness of the fixed point, assume that $x^*_1$ and $x^*_2$ are fixed points of $F$. Then
$$0\leq \norm{x^*_1-x^*_2}=\norm{F(x^*_1)-F(x^*_2)} \leq L \norm{x^*_1-x^*_2}.$$
$0\leq L<1$ implies that $\norm{x^*_1-x^*_2}=0$ and thus $x^*_1=x^*_2$.
\end{proof}

\begin{theorem}\label{thm:banach2}
Let $F:U\times V \subset X \times Y \rightarrow U$ with the following properties:
\begin{enumerate}[(a)]
	\item $X$ and $Y$ are Banach spaces
	\item $U\subset X$ is closed
	\item $F$ is of class $C^r$ for $r \geq 0$
	\item There is $L<1$ such that for all $y \in V$ we have $\norm{F(x,y)-F(\widetilde{x},y)} \leq L\norm{x- \widetilde{x}}$.
\end{enumerate}
Then for all $y\in V$ there exists a fixpoint $x^*(y)$ and the function $$x^*:V \rightarrow U, \quad y\mapsto x^*(y)$$  is of class $C^r$.
\end{theorem}

\begin{proof}
\begin{enumerate}[(1)]
\item \label{thm:banach2:step1} We will show that the function $x^*$ is continuous if $F$ is continuous.
\begin{multline*}
\norm{x^*(y+h)-x^*(y)} = \\
\begin{split}
&=\norm{F(x^*(y+h),y+h)-F(x^*(y),y+h)+F(x^*(y),y+h)-F(x^*(y),y)} \\
&\leq L \norm{x^*(y+h)-x^*(y)}+\norm{F(x^*(y),y+h)-F(x^*(y),y)}
\end{split}
\end{multline*}
It follows:
\begin{equation}\tag{$*$}\label{thm:banach2:stern}
\|x^*(y+h)-x^*(y)\|\leq \frac{1}{(1-L)}\|F(x^*(y),y+h)-F(x^*(y),y)\|
\end{equation}
Since $F$ is continuous and $(1-L)>0$ the continuity of $x^*$ is clear. 

\item  \label{thm:banach2:step2} We will show that $x^*$ is $C^1$ if $F$ is $C^1$. If $x^*$ is differentiable, it must satisfy
\[\frac{\ud x^*}{\ud y} (y)=\frac{\partial F}{\partial x}(x^*(y),y)\frac{\ud x^*}{\ud y} (y)+\frac{\partial F}{\partial y}(x^*(y),y).\]
This follows from differentiation of the equality $x^*(y)=F(x^*(y),y).$ To prove that $x^*$ is differentiable, we will apply the Banach Fixed Point Theorem to the operator 
$$T(.,y): L(Y,X) \rightarrow L(Y,X), \quad A\mapsto \frac{\partial F}{\partial x}(x^*(y),y)A+\frac{\partial F}{\partial y}(x^*(y),y).$$
Because $F$ is a contraction, we have $\norm{ \frac{\partial F}{\partial x} } \leq L$, and therefore
$$\norm{T(A,y)-T(\widetilde{A},y)}=\norm{ \frac{\partial F}{\partial x}(x^*(y),y)A- \frac{\partial F}{\partial x}(x^*(y),y)\widetilde{A}} \leq L \norm{A-\widetilde A }, $$
so $T$ is a contraction. By what we have shown in step (\ref{thm:banach2:step1}) there is a unique fixed point $A^*(y)$ of $T$ and $A^*$ is continuous in $y$. We want to show that $A^*$ is the derivative of $x^*$. With $u:= x^*(y+h)-x^*(y)$ this is equivalent to $\norm{u-A^*(y)h}=o(\norm{h})$. Because
\begin{equation*}
\norm{u- A^*(y)h} \leq \norm{\left( I-\frac{\partial F}{\partial x}(x^*(y),y)\right)^{-1}} \norm{\left(I-\frac{\partial F}{\partial x}(x^*(y),y)\right)(u- A^*(y)h)},
\end{equation*}
it is sufficient to show that 
$$\norm{\left(I-\frac{\partial F}{\partial x}(x^*(y),y)\right)(u- A^*(y)h)}=o(\norm{h}).$$ 
From the fixed point equation $A^*(y)=T(A^*(y),y)$ we have
$$\left( I- \frac{\partial F}{\partial x}(x^*(y),y) \right)A^*(y)=\frac{\partial F}{\partial y}(x^*(y),y),$$
and therefore
\begin{eqnarray*}
\left( I- \frac{\partial F}{\partial x}(x^*(y),y) \right)(u-A^*(y)h)=u-\frac{\partial F}{\partial x}(x^*(y),y)u-\frac{\partial F}{\partial y}(x^*(y),y)h\\=
F(x^*(y+h),y+h)-F(x^*(y),y))-\frac{\partial F}{\partial x}(x^*(y),y)u-\frac{\partial F}{\partial y}(x^*(y),y)h.
\end{eqnarray*}
The above expression is $o(\norm{u}+\norm{h})$. Since $F\in C^1$, it satisfies a Lipschitz condition in $y$, and one can read off equation \eqref{thm:banach2:stern} that $\norm{u(h)} = O(\norm{h})$. Therefore the above expression is also $o(\norm{h})$, and this concludes the proof. 

\item \label{thm:banach2:step3} We will show by induction that $x^* \in C^r$ if $F \in C^r$. The case $r=0$ has been treated in step (\ref{thm:banach2:step1}). For the inductive step, assume that $F \in C^{r+1}$. Then the map $T$ defined in step (\ref{thm:banach2:step2}) is $C^{r}$, and the derivative of $x^*$ is the fixed point of $T$. By the inductive assumption, the derivative of $x^*$ is $C^r$, and therefore $x^*$ is $C^{r+1}$. \qedhere
\end{enumerate}
\end{proof}

\begin{theorem}[Inverse Function Theorem] \label{thm:inversefunctiontheorem} \index{Inverse Function Theorem}
Let $X$ and $Y$ be Banach spaces, and let $U \subset X$ be an open set in $X$. Let $f:U \rightarrow Y$ be of class $C^r$, $r \geq 1$, $x_0 \in U$ and suppose that the Fr\'{e}chet derivative $f'(x_0)$ is a Banach space isomorphism. Then $f$ is a $C^r$-diffeomorphism of some neighbourhood of $x_0$ onto some neighbourhood of $f(x_0)$. 
\end{theorem}

\begin{proof}
First we note that it is enough to prove the theorem under the simplifying assumptions that 
$$X=Y, \quad u_0=f(x_0)=0 \quad \text{and} \quad f'(x_0)=\Id_X.$$ 
Indeed, we can replace $f$ by the map 
$$x \mapsto f'(x_0)^{-1} \circ \bigl(f(x+x_0)-f(x_0)\bigr).$$ 
With these assumptions we have
$$f(x)=y \Longleftrightarrow F(x,y):=f(x)-x-y=x.$$
The differential $\frac{\partial F}{\partial x}$ is independent of $y$ and vanishes at $(0,0)$, so there is $\epsilon >0$ such that 
$$\norm{\frac{\partial F}{\partial x}(x,y)}<\frac{1}{2} \quad \text{whenever} \quad \norm{x}<\epsilon.$$ 
Then 
$$F:B(\epsilon,0) \times B(\epsilon/2,0) \rightarrow B(\epsilon,0)$$
is well-defined and satisfies the hypothesis of the Banach Fixed Point Theorem \ref{thm:banach2} with $L=1/2$. The function $x^*(y)$ is of class $C^r$ and is the inverse of $f$. 
\end{proof}

We will need a generalization of the Constant Rank Theorem \ref{thm:finiteconstantranktheorem} to Banach spaces. 

\begin{definition} \index{split subspace} \index{complemented subspace}
A closed subspace $X_1$ of a Banach space $X$ is said to be split (or complemented) if there is a closed subspace $X_2 \subset X$ such that $X=X_1 \oplus X_2$.
\end{definition}

\begin{remark}
If $X$ and $Y$ are Banach spaces, $Y$ is finite-dimensional, and $f:X \rightarrow Y$, then the kernel and the image of $f$ are split. This situation will occur in the analysis of the endpoint map in Section \ref{sec:endpointmap}.
\end{remark}

\begin{theorem}[Constant Rank Theorem in Banach spaces] \label{thm:constantranktheorem} \index{Constant Rank Theorem ! in Banach spaces}
Let $X$ and $Y$ be Banach spaces, and let $U \subset X$ be an open set in $X$. Let $f:U \rightarrow Y$ be of class $C^r$, $r \geq 1$, $u_0 \in U$ and suppose that the Fr\'{e}chet derivative $f'(u_0)$ has closed split image $Y_1$ with closed complement $Y_2$ and split kernel $X_2$ with closed complement $X_1$. In addition, assume that for all $u$ in a neighbourhood of $u_0 \in U$, $f'(u)(X)$ is a closed subspace of $Y$ and $f'(u)|X_1:X_1 \rightarrow f'(u)(X)$ is a Banach space isomorphism. Then there exist open sets 
$$U_1 \subset Y_1 \oplus X_2, \quad U_2 \subset X, \quad V_1 \subset Y,\quad V_2 \subset Y$$
and there are $C^r$-diffeomorphisms 
$$\varphi:V_1 \rightarrow V_2 \quad \text{and} \quad \psi:U_1 \rightarrow U_2$$ 
such that $(\varphi \circ f \circ \psi)(x,e)=(x,0)$.
\end{theorem}

\begin{proof}\mbox{}\par
\begin{enumerate}[(1)]

\item \label{thm:constantranktheorem1}
Write $f=f_1 \times f_2$ with $f_i:U \rightarrow Y_i$ for $i=1,2$. Then 
$$\frac{\partial f_1}{\partial x_1}(u_0):X_1 \rightarrow Y_1$$
is a Banach space isomorphism. We define 
$$g:U \rightarrow Y_1 \oplus X_2, \quad g(x_1,x_2)=(f_1(x_1,x_2),x_2).$$
Then for $(h_{X_1},h_{X_2}) \in X_1 \oplus X_2$ we have
\begin{equation}\tag{$*$}\label{thm:constantranktheorem:eqn1}
g'(u)(h_{X_1},h_{X_2})=\left(\begin{array}{cc}\frac{\partial f_1}{x_1}(u) & \frac{\partial f_1}{x_2}(u) \\0 & \Id_{X_2}\end{array}\right) \left(\begin{array}{c}h_{X_1} \\ h_{X_2}\end{array}\right),
\end{equation}
and therefore $g'(u_0):X=X_1 \oplus X_2 \rightarrow Y_1 \oplus X_2$ is a Banach space isomorphism. So by the Inverse Function Theorem \ref{thm:inversefunctiontheorem} $g$ has a local inverse $\psi:U_1 \rightarrow U_2 \subset U$ of class $C^r$, $\psi^{-1}=g|U_2$. Then $(f_1 \circ \psi)(y_1,x_2)=y_1$ for $(y_1,x_2) \in U_1$.

\item \label{thm:constantranktheorem2}
We claim that 
$$(f \circ \psi)' (y_1,x_2) | Y_1 \times \{0\} : Y_1 \times \{0\} \rightarrow (f \circ \psi)'(y_1,x_2)(Y_1 \oplus X_2)$$
is a Banach space isomorphism for $(y_1,x_2) \in U_1$. To prove the claim, we first notice by looking at \eqref{thm:constantranktheorem:eqn1} that 
$$g'(u) | X_1 \times \{0\}:X_1 \times \{0\} \rightarrow Y_1 \times \{0\}$$
is a Banach space isomorphism. Therefore its inverse maps $Y_1 \times \{0\}$ bijectively to $X_1\times \{0\}$. Furthermore, the hypothesis of the theorem says that $f'(\psi(y_1,x_2))$ is a Banach space isomorphism when restricted to $X_1 \times \{0\}$, and the claim follows by the chain rule 
$$(f \circ \psi)'(y_1,x_2)=f'(\psi(y_1,x_2)) \circ \psi'(y_1,x_2).$$ 

\item \label{thm:constantranktheorem3}
Let $(h_{Y_1},h_{X_2}) \in Y_1 \oplus X_2$. Then because of (\ref{thm:constantranktheorem2}) there is $\tilde h_{Y_1}$ such that
\begin{equation} \tag{$**$} \label{thm:constantranktheorem:eqn2}
(f \circ \psi)'(y_1,x_2)(h_{Y_1},h_{X_2})=(f \circ \psi)'(y_1,x_2)(\tilde h_{Y_1},0).
\end{equation}
We see from (\ref{thm:constantranktheorem1}) that
$$(f_1 \circ \psi)'(y_1,x_2)(h_{Y_1},h_{X_2})=h_{Y_1},$$ 
so $\tilde h_{Y_1}$ equals $h_{Y_1}$. But then \eqref{thm:constantranktheorem:eqn2} implies that 
$$\frac{\partial f \circ \psi}{\partial x_2}(y_1,x_2)=0,$$
so $f \circ \psi$ does not depend on the variable $x_2$, and we can write 
$$\tilde f:P_{Y_1}(V) \subset Y_1 \rightarrow Y, \quad \tilde f(y_1) = f \circ \psi(y_1,x_2)$$ 
where $P_{Y_1}:Y_1 \oplus X_2 \rightarrow Y_1$ is the projection. 

\item \label{thm:constantranktheorem4}
Let $y_0=P_{Y_1}(\psi^{-1}(u_0))$. The derivative of $\tilde f$ at $y_0$ is injective and has closed split image $Y_1$ with closed complement $Y_2$. Define 
$$\tilde g:P_{Y_1}(V) \times Y_2 \subset Y_1 \oplus Y_2 \rightarrow Y_1 \oplus Y_2, \quad \tilde g(y_1,y_2)=\tilde f(y_1)+(0,y_2).$$
Then the derivative of $\tilde g$ at $(y_0,0)$ is a Banach space isomorphism, and therefore by the Inverse Function Theorem \ref{thm:inversefunctiontheorem} there is a $C^r$-diffeomorphism $\varphi:V_1 \rightarrow V_2$ with $V_1, V_2 \subset Y$ and $\varphi^{-1}=\tilde g | V_2$. Because $\tilde f(y_1)=\tilde g(y_1,0)$, we have
$$(\varphi \circ \tilde f)(y_1)=\varphi(\tilde g(y_1,0))=(y_1,0) \in Y_1 \oplus Y_2$$
for all $y_1$ such that $(y_1,0) \in V_2$. Therefore $(\varphi \circ f \circ \psi)(y_1,x_2)=(y_1,0)$. \qedhere
\end{enumerate}
\end{proof}

\subsubsection{Superposition operators} 

\begin{definition}\label{def:superpositionoperator} 
\index{superposition operator} \index{composition operator} \index{Nemytskij operator}
Let $f :\Omega \times X_1 \rightarrow X_2$ be a function. Then the superposition operator associated to $f$ is a mapping $F$ that assigns to each function $x : \Omega \rightarrow X_1$ a function $Fx: \Omega \rightarrow X_2$ defined as 
$$Fx(\omega)=f(\omega,x(\omega)).$$ 
$F$ is also called composition operator or Nemytskij operator. 
\end{definition}

The book of \citet{Appell} gives a very good overview of superposition operators acting on different spaces of functions. Our interest lies in the situation where $F$ acts on continuous functions. The results presented here are generalizations of \citet[chap.~6]{Appell}, where only the case $X_1=X_2=\mathbb R$ is treated. 

\begin{definition}
If $\Omega$ is a compact metric space and $X$ is a Banach space, then the set $C(\Omega; X)$ of continuous functions from $\Omega$ to $X$ is a Banach space when equipped with the norm 
\begin{equation*}
\norm{x}_{C(\Omega;X)}=\sup_{s \in \Omega} \norm{x(s)}_X.
\end{equation*}
\end{definition}

\begin{lemma}\label{lem:continuoussuperpositionoperator}
Let $X_1$ and $X_2$ be Banach spaces, $\Omega$ a compact metric space, $U$ an open subset of $X_1$, and $f:\Omega \times U \rightarrow X_2$.
If $f$ is continuous, then its associated superposition operator $F:C(\Omega;U) \rightarrow C(\Omega;X_2)$ is continuous.
\end{lemma}

\begin{proof}
If $f$ is continuous, then clearly $F$ maps continuous functions to continuous functions. It remains to show that $F$ is continuous at every $x \in C(\Omega;U)$. Because $\Omega$ is compact and $x$ is continuous, the graph of $x$ is a compact subspace of $\Omega \times U$ equipped with the product metric 
\begin{displaymath}
d \left( (\omega_1,x_1),(\omega_2,x_2) \right) = \max \left(d(\omega_1,\omega_2), \norm{x_1-x_2}_{X_1} \right).
\end{displaymath} 
$f$ is continuous, so around every point $p$ in the graph of $x$ we find a metric ball $B_{\delta(p)}(p)$ with radius $\delta(p)$ such that whenever $q$ lies in this ball, we have 
$$\norm{f(p)-f(q)}_{X_2} < \epsilon/2.$$ 
Finitely many balls $B_{\delta(p)/2}(p)$ cover the graph of $x$, say the balls $B_{\delta(p_i)/2}(p_i)$ for $i=1, \dots, n$. Let $\delta$ be the smallest radius of these balls, i.e. 
\begin{displaymath}
\delta=\min \left\{\delta(p_i)/2: i=1, \dots, n\right\}.
\end{displaymath} 
Now we claim that if $y$ is a function in $C(\Omega;U)$, we have
$$\norm{x-y}_{C(\Omega;X_1)} < \delta \ \Longrightarrow \ \norm{Fx-Fy}_{C(\Omega;X_2)} < \epsilon.$$ 
To prove this claim, take a point $p_y=(\omega,y(\omega))$ in the graph of $y$. Because $p_x=(\omega,x(\omega))$ lies in the graph of $x$ and the graph of $x$ is covered by the balls $B_{\delta(p_i)/2}(p_i)$, there is a point $p_{i_0}$ in the graph of $x$ such that $d(p_x, p_{i_0})< \delta(p_{i_0})/2$. Now, 
\begin{displaymath}
d(p_y,p_{i_0}) \leq d(p_y,p_x)+d(p_x,p_{i_0}) < \delta + \delta(p_{i_0})/2 \leq \delta(p_{i_0}).
\end{displaymath} 
So both $p_y$ and $p_x$ lie in the same ball $B_{\delta(p_{i_0})}(p_{i_0})$. Therefore 
\begin{displaymath}
\norm{f(p_x)-f(p_y)}_{X_2} \leq \norm{f(p_x)-f(p_{i_0})}_{X_2} + \norm{f(p_{i_0})-f(p_y)}_{X_2} < \epsilon/2 + \epsilon/2 = \epsilon.
\end{displaymath}
Since this is true for all points $p_x=(\omega,x(\omega))$ and $p_y=(\omega,y(\omega))$, $\omega \in \Omega$, we get the desired inequality $\norm{Fx-Fy}_{C(\Omega;X_2)}<\epsilon$.
\end{proof}

\begin{lemma}\label{lem:differentiablesuperpositionoperator}
Let $X_1$ and $X_2$ be Banach spaces, $\Omega$ a compact metric space, $U$ an open subset of $X_1$, and $f:\Omega \times U \rightarrow X_2$.
If $f$ is continuous and has a continuous partial derivative $\frac{\partial f}{\partial x}$, then its associated superposition operator $F:C(\Omega;U) \rightarrow C(\Omega;X_2)$ is of class $C^1$. 
\end{lemma}

\begin{proof}
We claim that the Fr\'{e}chet derivative of $F$ at $x \in C(\Omega; U)$ is given by
\begin{displaymath}
F'(x)(h)(\omega) = \frac{\partial f}{\partial x} (\omega, x(\omega)) h(\omega)\ \text{ or equivalently } \ F'(x) = M(G(x)),
\end{displaymath} 
where M and G are defined as follows:
\begin{displaymath}
M : C(\Omega;L(X_1;X_2)) \rightarrow L(C(\Omega;X_1); C(\Omega;X_2)), \quad M(g)(h)(\omega)=g(\omega)(h(\omega)).
\end{displaymath}
It is easy to see that $M$ is a well-defined bounded linear operator. Therefore $M$ is smooth. 
\begin{displaymath}
G : C(\Omega;U)  \rightarrow  C(\Omega;L(X_1;X_2))
\end{displaymath} 
is the superposition operator associated to the continuous function 
\begin{displaymath}
\frac{\partial f}{\partial x}: \Omega \times U \rightarrow L(X_1;X_2).
\end{displaymath}
By Lemma \ref{lem:continuoussuperpositionoperator} we know that $G$ is well-defined and continuous. To prove that $M(G(x))$ really is the Fr\'{e}chet derivative of $F$ at $x$, we have to show that
\begin{enumerate}[(1)]
\item \label{lem:differentiablesuperpositionoperator1}
For all $x \in C(\Omega;U)$, $M(G(x))$ is a bounded linear operator from $C(\Omega;X_1)$ to $C(\Omega;X_2)$.
\item  \label{lem:differentiablesuperpositionoperator2}
$\norm{F(x+h)-F(x)-M(G(x))(h)}_{C(\Omega;X_2)} = o \left(\norm{h}_{C(\Omega;X_1)}\right)$.
\end{enumerate}
(\ref{lem:differentiablesuperpositionoperator1}) is clear from the definition of $M$. (\ref{lem:differentiablesuperpositionoperator2}) follows from the Taylor formula for $f$:
\begin{multline*}
\norm{F(x+h) -F(x)-M(G(x))(h)}_{C(\Omega;X_2)} = \\
\begin{split}
&=\sup_{\omega \in \Omega} \norm{f(\omega,x(\omega)+h(\omega))-f(\omega,x(\omega))-
\frac{\partial f}{\partial x} (\omega, x(\omega)) h(\omega)}_{X_2}  \\
&=\sup_{\omega \in \Omega} \norm{\int_0^1 \frac{\partial f}{\partial x} (\omega, x(\omega)+t h(\omega)) h(\omega)\thinspace dt -
\frac{\partial f}{\partial x} (\omega, x(\omega)) h(\omega)}_{X_2}  \\
&=\sup_{\omega \in \Omega} \norm{\int_0^1 \left(\frac{\partial f}{\partial x} (\omega, x(\omega)+t h(\omega))-\frac{\partial f}{\partial x} (\omega, x(\omega)) \right)h(\omega)\medspace dt}_{X_2} \\
&\leq \sup_{\omega \in \Omega} \sup_{0\leq t \leq 1} \norm{\frac{\partial f}{\partial x} (\omega, x(\omega)+t h(\omega))-\frac{\partial f}{\partial x} (\omega, x(\omega)) }_{L(X_1;X_2)} \norm{h(\omega)}_{X_1}  \\
&\leq \norm{h}_{C(\Omega;X_1)} \sup_{0 \leq t \leq 1} \norm{G(x+t h)-G(x)}_{C(\Omega;L(X_1;X_2))}\\
&=o \left(\norm{h}_{C(\Omega;X_1)}\right), \text{ because $G$ is continuous.}
\end{split}
\end{multline*}
It remains to show that $x \mapsto F'(x)$ is continuous. This is clear since $G$ is continuous and $M$ is smooth. 
\end{proof}

\begin{theorem}\label{thm:smoothsuperpositionoperator}
Let $X_1$ and $X_2$ be Banach spaces, $\Omega$ a compact metric space, $U$ a compact subspace of $X_1$ and $f:\Omega \times U \rightarrow X_2$.
If $f$ is continuous and has continuous partial derivatives $\frac{\partial^k f}{\partial x^k}$ for $0 \leq k \leq n$, then its associated superposition operator $F:C(\Omega;U) \rightarrow C(\Omega;X_2)$ is of class $C^n$. 
\end{theorem}

\begin{proof}
The case $n=0$ has been proven in Lemma \ref{lem:continuoussuperpositionoperator}, and the case $n=1$ in Lemma \ref{lem:differentiablesuperpositionoperator}. We will prove the theorem by induction. For the inductive step, let $f$ have continuous partial derivatives for $1 \leq k \leq n+1$. By Lemma \ref{lem:differentiablesuperpositionoperator}, $F$ is of class$C^1$ and its derivative is given by $F'=M \circ G$. $G$ is the superposition operator associated to the function $\frac{\partial f}{\partial x}$, which is continuous and has $n$ continuous partial derivatives. By the inductive assumption, $G$ is of class $C^n$. Because $M$ is $C^\infty$, $F'=M \circ G$ is of class $C^n$. So $F$ is of class $C^{n+1}$. 
\end{proof}

\subsubsection{The endpoint map} \label{sec:endpointmap} 

In all of section \ref{sec:endpointmap}, we will use the following assumptions: 
\begin{enumerate}[(a)]
\item \label{endpointmap:assumption1} 
$M$ is an open connected subset of $\mathbb R^n$. 
\item \label{endpointmap:assumption2} 
$H$ is spanned by a finite number of vectorfields $X_1, \dots, X_k \in \mathfrak X(M)$.
\item \label{endpointmap:assumption3}  
There is $K \in \mathbb R$ such that $\norm{X_i(x)}_{\mathbb R^n} \leq K$ for all $x \in M$. 
\item \label{endpointmap:assumption4}  
There is $L \in \mathbb R$ such that $\norm{X_i(x)-X_i(y)}_{\mathbb R^n} \leq L$ for all $x,y \in M$. 
\end{enumerate}

\begin{theorem} \label{thm:controlledcurve}
Under the assumptions (\ref{endpointmap:assumption1}) to (\ref{endpointmap:assumption4}) there is an open set $U$, 
$$U \subset L^1([0,1];\mathbb R^k) \times M, \quad \{0\} \times M \subset U$$ 
such that for each $(u,x) \in U$ the equation 
\begin{equation}\tag{$*$}\label{eqn:controlledcurve} 
\dot \gamma(t)=\sum_{i=1}^k u_i(t) X_i(\gamma(t)), \quad \gamma(0)=x
\end{equation}
has a unique solution $\gamma \in C([0,1];M)$ which depends smoothly on $u$ and $x$. 
\end{theorem}

\begin{proof}\indent\par
\begin{enumerate}[(1)]

\item \label{thm:controlledcurve:step1} For the moment, let us work on an arbitrary interval $[a,b]$ contained in $[0,1] \subset \mathbb R$. We need some conventions about the norms used. Define 
\begin{equation*}\begin{split}
\norm{x}_{\mathbb R^n}&=\sum_{i=1}^n \abs{x_i}\\
\norm{u}_{L^1([a,b];\mathbb R^k)}&=\sum_{i=1}^k \int_a^b \abs{u_i(s)} \ud s \\
\norm{u}_{C([a,b];\mathbb R^k)} &=\sum_{i=1}^k \max_{s \in [a,b]} \abs{u_i(s)}.
\end{split}\end{equation*}

\item \label{thm:controlledcurve:step2} Let $I$ denote the integral operator 
\begin{equation*}
I: \begin{cases}
L^1([a,b];\mathbb R^k) & \rightarrow  \quad C([a,b];\mathbb R^k) \\
u & \mapsto  \quad (t \mapsto \int_0^t u(s) \ud s).
\end{cases}
\end{equation*}
Then $I$ is a linear operator with operator norm equal to one, so $I$ is smooth. Define
\begin{equation*}
F: \begin{cases}
C([a,b];M) \times L^1([a,b];\mathbb R^k) \times M & \rightarrow  \quad C([a,b];\mathbb R^n) \\
(\gamma, u, x) & \mapsto  \quad x + I(\sum_{i=1}^k u_i\ X_i \circ \gamma).
\end{cases}
\end{equation*}
Note that $F$ is defined on an open set because $C([a,b];M)$ is open in $C([a,b];\mathbb R^n)$ whenever $M \subset \mathbb R^n$ is open. By Theorem \ref{thm:smoothsuperpositionoperator}, $F$ is smooth. 

\item \label{thm:controlledcurve:step3}We can rewrite the differential equation \eqref{eqn:controlledcurve} as an equivalent integral equation $\gamma=F(\gamma,u,x)$. 

\item \label{thm:controlledcurve:step4} Let $\gamma_0=F(\gamma_0,u_0,x_0)$ be a solution on the interval $[a,b]$ with $(\gamma_0, u_0, x_0)$ in the domain of definition of $F$. We claim that under the assumption 
$$\norm{u_0}_{L^1([a,b];\mathbb R^k)}<\frac{1}{3 L}$$ 
we can find a neighbourhood $U$ of $(u_0,x_0)$ in $L^1([a,b];\mathbb R^k) \times M$ such that  for all $(u,x)\in U$ the equation $\gamma=F(\gamma,u,x)$ has a unique solution $\gamma \in C([a,b];M)$. To prove the claim, we choose $\epsilon >0$ small enough such that $B(\epsilon, \gamma_0) \subset C([a,b];M)$.  Then we define 
\begin{equation*}
G: \begin{cases}
\overline{B(\epsilon, \gamma_0)} \times \left( B(\frac{\epsilon}{3 K},u_0) \cap B(\frac{1}{3 L},0) \right) \times B(\frac{\epsilon}{3},x_0)  & \rightarrow  \quad \overline{B(\epsilon, \gamma_0)}\\
(\gamma, u, x) & \mapsto  \quad F(x,\gamma,u).
\end{cases}
\end{equation*}
First, we show that the image of $G$ really is a subset of $\overline{B(\epsilon, \gamma_0)}$. 
\begin{equation*}\begin{split}
&\norm{G(\gamma,u,x)-\gamma_0}_{C([a,b];\mathbb R^n)} = \norm{G(\gamma,u,x)-G(\gamma_0,u_0,x_0)}_{C([a,b];\mathbb R^n)}\\
{}
&\quad \leq \norm{x-x_0}_{\mathbb R^n} + 
\norm{ \sum_{i=1}^k u_i\ X_i \circ \gamma - u_{0,i}\ X_i \circ \gamma_0}_{L^1([a,b];\mathbb R^n)}\\ 
{}
&\quad \leq \norm{x-x_0}_{\mathbb R^n} +
\sum_{i=1}^k \norm{u_i\ X_i \circ \gamma - u_{0,i}\ X_i \circ \gamma}_{L^1([a,b];\mathbb R^n)}+\\
&\quad \quad \quad \quad \quad \quad \quad\ +\sum_{i=1}^k \norm{u_{0,i}\ X_i \circ \gamma - u_{0,i}\ X_i \circ \gamma_0}_{L^1([a,b];\mathbb R^n)}\\
{}
&\quad \leq \norm{x-x_0}_{\mathbb R^n} + \norm{u-u_0}_{L^1([a,b];\mathbb R^k)} K + \norm{u_0}_{L^1([a,b];\mathbb R^k)} L \norm{\gamma-\gamma_0}_{C([a,b];\mathbb R^n)}\\
{}
&\quad \leq \frac{\epsilon}{3}+\frac{\epsilon}{3}+\frac{\epsilon}{3}=\epsilon.
\end{split}\end{equation*}

Furthermore, $G$ satisfies the contraction property
\begin{equation*}\begin{split}
\norm{ G(\gamma,u,x)-G(\tilde \gamma,u,x)}_{C([a,b];\mathbb R^n)} 
\leq \sum_{i=1}^k \norm{u_i (X_i \circ \gamma - X_i \circ \tilde\gamma)}_{L^1([a,b];\mathbb R^n)} \\
\leq \norm{u}_{L^1([a,b];\mathbb R^k)} L \norm{\gamma - \tilde\gamma}_{C([a,b];\mathbb R^n)} \leq \frac{1}{3} \norm{\gamma - \tilde\gamma}_{C([a,b];\mathbb R^n)}
\end{split}\end{equation*}
for all $(\gamma,u,x)$ and $(\tilde \gamma,u,x)$ in the domain of $G$. Therefore we can apply the Banach Fixed Point Theorem \ref{thm:banach2} to $G$, and this proves our claim with 
$$U= \left( B(\frac{\epsilon}{3 K},u_0) \cap B(\frac{1}{3 L},0) \right) \times B(\frac{\epsilon}{3},x_0).$$ 

\item \label{thm:controlledcurve:step5} Now we will work on the interval $[0,1]$ and drop the assumption that 
$$\norm{u_0}_{L^1([0,1];\mathbb R^k)}<\frac{1}{3 L}.$$ 
Let $\gamma_0=F(\gamma_0,u_0,x_0)$ be a solution on the interval $[0,1]$ with $(\gamma_0, u_0, x_0)$ in the domain of definition of $F$. We claim that we can find a neighbourhood $U$ of $(u_0,x_0)$ in $L^1([0,1];\mathbb R^k) \times M$ such that  for all $(u,x)\in U$ the equation $\gamma=F(\gamma,u,x)$ has a unique solution $\gamma \in C([0,1];M)$. To prove the claim, we choose $\epsilon >0$ small enough such that $B(\epsilon, \gamma_0) \subset C([a,b];M)$. Let $0=t_0<\dots<t_N=1$ such that on each interval $[t_j,t_{j+1}]$ 
$$\norm{u_0|_{[t_j,t_{j+1}]}}_{L^1([t_j,t_{j+1}];\mathbb R^n)}<\frac{1}{3 L}$$
holds. Then we can apply the result obtained in (\ref{thm:controlledcurve:step4}) to each of the intervals $[t_j,t_{j+1}]$ with 
$$a=t_j, b=t_{j+1}, \gamma_0=\gamma_0|_{[t_j,t_{j+1}]}, u_0=u_0|_{[t_j,t_{j+1}]}, x_0=\gamma(t_j).$$ 
We want to concatenate the solutions on these intervals to get a solution on $[0,1]$, so we have to make sure that the endpoint of a solution curve on $[t_j,t_{j+1}]$ lies in $B(\frac{\epsilon}{3},\gamma(t_{j+1}))$. This is the fact for 
$$U=\left(B(\frac{\epsilon}{3^N},u_0) \cap B(\frac{1}{3 L},0) \right) \times B(\frac{\epsilon}{3^N},x_0).$$

The solution on each interval $[t_j,t_{j+1}]$ depends smoothly on the control and initial value. Because the operator assigning to each curve its endpoint is a linear, continuous operator $C([t_j,t_{j+1}];M) \rightarrow M$, it is smooth, and therefore the concatenation of the solutions curves depends smoothly on $x_0$ and $u_0$. \qedhere

\end{enumerate}
\end{proof}

The remaining part of this section is taken from \citet{Bellaiche}. 

\begin{definition} \index{endpoint map} \index{controlled vectorfield} \index{evolution of controlled vectorfield}
In the situation of Theorem \ref{thm:controlledcurve}, we define
$$\Evol^u_t(x)=\gamma(t) \quad \text{and} \quad \End_x(u)=\gamma(1)$$ 
for $t \in [0,1]$, $(u,x)\in U$ and a curve $\gamma \in C([0,1];M)$ with initial value $x$ controlled by $u$. We will call $\End_x$ the endpoint map and $\Evol^u_t$ the evolution of the controlled vectorfield $\sum_{i=1}^k u_i X_i$.  
\end{definition}

\begin{remark}\label{rem:Endpoint}
$\End_x$ is smooth for each $x \in M$. This is true because the evaluation operator 
$$C([0,1];\mathbb R^n) \rightarrow \mathbb R^n, \quad \gamma \mapsto \gamma(1)$$ 
is linear and bounded, and therefore smooth, and because furthermore $\gamma$ depends smoothly on its control $u$ and initial value $x$ by Theorem \ref{thm:controlledcurve}. 
\end{remark}

\begin{definition} 
\index{control ! concatenation} \index{control ! inverse} \index{inverse control} \index{concatenation of controls}
For $u,v \in L^1([0,1];\mathbb R^k)$ and $0<s<1$, let $u *_s v \in L^1([0,1];\mathbb R^k)$ be the control function defined by
\begin{equation*}
(u *_s v)(t) = 
\begin{cases}
\frac{1}{s}\ u(\frac{t}{s}) & \text{if $0 \leq t <s$}\\
\frac{1}{1-s}\ v(\frac{t-s}{1-s}) & \text{if $s \leq t \leq 1$.}
\end{cases}
\end{equation*}
In many cases we will use $s=\frac{1}{2}$, so we define $u*v=u*_{\frac{1}{2}}v$ for the sake of a simpler notation. Furthermore, we define 
\begin{equation*}
\check u (t) = - u(1-t) \quad \text{for $0 \leq t \leq 1$}.
\end{equation*}
$u*_s v$ will be called the concatenation of $u$ and $v$, and $\check u$ the inverse control of $u$. 
\end{definition}

\begin{remark} 
It is easy to check the following properties:
\begin{enumerate}[(1)]
\item $\norm{u*_s v}_{L^1([0,1];\mathbb R^k)}=\norm{u}_{L^1([0,1];\mathbb R^k)}+\norm{v}_{L^1([0,1];\mathbb R^k)}$. 
\item $\Evol^{u*_s v}_1=\Evol_1^v \circ \Evol_1^u$
\item $\Evol_1^u$ is a local diffeomorphism with inverse $\Evol_1^{\check u}$. 
\end{enumerate}
\end{remark}

\begin{definition} \label{def:normalcontrol} \index{normal control} \index{control ! normal}
Let $r$ denote the maximal rank of $\End_x$. A control $u \in U$ is called normal if the rank of $\End_x$ at $u$ is equal to $r$.
\end{definition}

\begin{lemma}\label{lem:normalcontrol}
If $u$ is a normal control, so is $u *_s v$ for any $0<s<1$.
\end{lemma}

\begin{proof}
The tangent space in $L^1([0,1];\mathbb R^k)$ at $u*_s v$ is given by 
$$\evaluate{\frac{\ud}{\ud t}}{0} (u*_s v)+t h, \quad h \in L^1([0,1];\mathbb R^k).$$
The tangent vectors of the form $$\evaluate{\frac{\ud}{\ud t}}{0} (u+t h)*_s v, \quad h \in L^1([0,1];\mathbb R^k)$$ constitute a vector subspace of the tangent space at $u*_s v$. We have
\begin{equation*}\begin{split}
T \End_x \left(\evaluate{\frac{\ud}{\ud t}}{0} (u+t h)*_s v\right)&=\evaluate{\frac{\ud}{\ud t}}{0} \End_x \left((u+t h)*_s v\right)=\\
&=\evaluate{\frac{\ud}{\ud t}}{0} \Evol_1^v \circ \End_x \left( u+t h\right)=\\
&=T \Evol_1^v \circ T \End_x \left(\evaluate{\frac{\ud}{\ud t}}{0} u+t h\right).
\end{split}\end{equation*}
The vectors $$\evaluate{\frac{\ud}{\ud t}}{0} u+t h, \quad h \in L^1([0,1];\mathbb R^k)$$ constitute the tangent space $T_u L^1([0,1];\mathbb R^k))$. Let $r$ denote the maximal rank of $\End_x$. Then because we assumed that $u$ is a normal control, the dimension of $T \End_x(T_u L^1([0,1];\mathbb R^k))$ is $r$. $\Evol_1^v$ is a local diffeomorphism, and this concludes the proof. 
\end{proof}

\begin{lemma}\label{lem:normallyaccessible}
For every point $y$ in the image of $\End_x$ there is a normal control steering $x$ to $y$. 
\end{lemma}

\begin{proof}
Suppose $y$ is attained from $x$ by means of a control $u$. Choose any normal control $v$. Then the control $v*\check v*u$ steers $x$ to $y$ and is normal by Lemma \ref{lem:normalcontrol}.
\end{proof}

\begin{theorem}\label{thm:endpointmapinitialsubmanifold}
For $x \in M$, the image of $\End_x$ is an initial submanifold, and its inclusion into $M$ is an initial mapping.
\end{theorem}

\begin{proof}\mbox{}\par
\begin{enumerate}[(1)]
\item Let $N$ denote the set of controls in the domain of $\End_x$ with maximal rank $r$. By Lemma \ref{lem:normallyaccessible}, we can restrict $\End_x$ to $N$ without loosing any points in the image of $\End_x$, and we will do this without change of notation. 

\item Now, $\End_x$ is a smooth map with constant rank $r$. The Constant Rank Theorem \ref{thm:constantranktheorem} gives us for every $u \in N$ a neighbourhood $U$ of $u$ such that $\End_x(U)$ is a small $r$-dimensional submanifold of $M$. To make $\End_x(N)$ an immersed submanifold, we give it the coarsest topology such that $\End_x$ is an open mapping, and we use the charts of the small submanifolds $\End_x(U)$ as charts. We need to check that each set $\End_x(U)$ is open in $\End_x(N)$, which is trivial, and that $\End_x(U_1) \cap \End_x(U_2)$ is open. Let $$y=\End_x(u_1)=\End_x(u_2) \in \End_x(U_1) \cap \End_x(U_2).$$ 
We look at the control $u_3 := u_1 * u_2^{-1} * u_2$, for which the Constant Rank Theorem gives us a neighbourhood $U_3$ of $u_3$ such that $\End_x(U_3)$ is a submanifold. By making $U_1$ and $U_2$ small enough, both $U_1 * u_2^{-1} * u_2$ and $u_1 * u_2^{-1} * U_2$ can be made a subset of $U_3$. So both $\End_x(U_1)=\End_x(U_1 * u_2^{-1} * u_2)$ and $\End_x(U_2)=\End_x(u_1 * u_2^{-1} * U_2)$ are $r$-dimensional submanifolds of $M$ lying in the $r$-dimensional manifold $\End_x(U_3)$. Hence their intersection is open in $\End_x(U_3)$, and therefore also open in $\End_x(N)$. From these considerations we know that $(\End_x(N),i)$ is an immersed submanifold, where $i:\End_x(N)\rightarrow M$ is the inclusion. 

\item The rest of the proof is basically the same as the proof of Theorem \ref{thm:globalsussmann}. We will construct initial submanifold charts for $\End_x(N)$. For $y = \End_x(u) \in \End_x(N)$, by the Constant Rank Theorem \ref{thm:constantranktheorem}, there is a chart $(U,\phi)$ centered at $u \in N$ and a chart $(V,\psi)$ centered at $y \in M$ such that 
$$(\psi \circ \End_x \circ \phi^{-1})(t^1, \dots, t^r,e)=(t^1, \dots, t^r,0, \dots, 0).$$ 
Then 
$$f:(t^1, \dots,t^n) \mapsto \End_{\psi^{-1}(0,\dots,0,t^{r+1}, \dots, t^n)}\circ \phi^{-1}(t^1, \dots, t^r, 0, \dots, 0)$$ 
is a diffeomorphism near zero. Let $(W,\chi)$ be the chart given by $f^{-1}$, suitably restricted. For all $t^1, \dots, t^n$ such that both expressions make sense, we have
$$f(t^1, \dots,t^n) \in \End_x(N) \Longleftrightarrow f(0, \dots, 0, t^{r+1}, \dots, t^n) \in \End_x(N).$$
From this we see that $\End_x(N) \cap W$ is the disjoint union of connected sets of the form 
$$\{w \in W: \chi^{r+1}(w), \dots, \chi^n(w) = \text{constant}\}.$$ Since $\End_x(N)$ is a connected immersed submanifold of $M$, it is second countable, and only a countable set of constants can appear in the description of $\End_x(N) \cap W$ given above. Therefore $\chi$ is a distinguished chart as in Theorem \ref{thm:globalsussmann}. Furthermore, we see that 
$$\chi(C_x(\End_x(N) \cap W))=\chi(W)\cap \mathbb R^r \times \{0\},$$ 
so $\chi$ is an initial submanifold chart. Because we can find such a chart around every point in $\End_x(N)$, $\End_x(N)$ is an initial submanifold. 

\item According to Lemma \ref{lem:initialsubmanifolduniquestructure}, there is a unique smooth manifold structure that makes $\End_x(N)$ an immersed submanifold with the property that the inclusion is an initial mapping. It can be seen from the proof of this Lemma that charts for this structure are given by restricting the initial submanifold charts $(W,\chi)$ to $C_x(\End_x(N)\cap W)$. Restricting $\chi$ to $C_x(\End_x(N)\cap W)$ yields the inverse of the map 
\begin{equation*}
(t^1, \dots, t^r) \mapsto \End_x \circ \phi^{-1}(t^1, \dots, t^r,0,\dots,0),
\end{equation*}
and this map already is a chart of $\End_x(N)$. \qedhere
\end{enumerate}
\end{proof}

\begin{theorem}\label{thm:accessiblesetandendpointmap}
For $x\in M$, the accessible set $Acc(x)$ equals the image of $\End_x$. 
\end{theorem}

\begin{proof}
Let $L$ denote the image of $\End_x$. Of course, $L$ is horizontally pathwise connected, so we have $L \subset Acc(x)$. We will show that $Acc(x) \subset L$ in a few steps. For the sake of a simpler notation let $\norm{\cdot}$ denote $\norm{\cdot}_{L^1([0,1];\mathbb R^k)}$. 
\begin{enumerate}[(1)]

\item \label{thm:accessiblesetandendpointmap1}We claim that for any function $u \in L^1([0,1];\mathbb R^k)$ we have $$\lim_{s \rightarrow 1} \norm{u *_s 0 - u}=0,$$ where $0$ denotes the zero control. This is clear for continuous functions $u$ by dominated convergence, and it still holds for arbitrary $u$ by a density argument: Indeed, for $\epsilon >0$ we can find a continuous function $v$ with $\norm{u-v}<\epsilon/3$. Then 
\begin{equation*} \begin{split}
\norm{u *_s 0-u} & \leq \norm{u *_s 0-v*_s 0} + \norm{v*_s 0 -v}+\norm{v-u} \\
& = \norm{u -v}+ \norm{v*_s 0 -v}+\norm{v-u} \leq \frac{\epsilon}{3} + \frac{\epsilon}{3}+\frac{\epsilon}{3}=\epsilon
\end{split}\end{equation*}
for $s$ close enough to one, and this proves the claim. 

\item \label{thm:accessiblesetandendpointmap2}For $j=1,\dots,k$, let $e_j$ be the unit vectors in $\mathbb R^k$, and let $u$ be a normal control. From 
\begin{equation*}\begin{split}
\norm{u*_s ((1-s)\ e_j)-u} & \leq \norm{u*_s ((1-s)\ e_j)-u*_s 0}+\norm{u*_s 0 - u}\\
& =\norm{(1-s)\ e_j}+\norm{u*_s 0 - u}
\end{split}\end{equation*}
we see that $u*_s((1-s)\ e_j)$ is still a normal control for $s$ close to one because the set of normal controls is open. But $\End_x(u*_s((1-s)\ e_j))=Fl^{X_j}_{1-s}(\End_x(u))$. Therefore we can conclude that $L$ is stable under the flows of $X_j$. The flows of $X_j$ are smooth mappings taking values in $L$. Because $i:L \rightarrow M$ is an initial mapping, the flows are also smooth as mappings into $L$, and therefore the vectorfields $X_j$ are tangent to $L$. We conclude that the tangent space of $L$ contains $H$. 

\item \label{thm:accessiblesetandendpointmap3}We claim that a horizontal path never leaves $L$. The proof is exactly the same as the one given in Theorem \ref{thm:accessiblesetisintegralmanifold}. \qedhere
\end{enumerate}
\end{proof}

With Theorem \ref{thm:accessiblesetandendpointmap} we are able to prove Chow's Theorem. However, we will first prove the slightly stronger

\begin{theorem}\label{thm:endpointmapisopen}
If $H$ is a bracket generating distribution on $M$, then the endpoint map is open and its maximal rank equals the dimension of $M$. 
\end{theorem}

\begin{proof}
We have to show that for an open subset $U$ of the domain of definition of $\End_x$ and for $y =\End_x(u) \in \End_x(U)$, there is a neighbourhood of $y$ that is contained in $\End_x(U)$. 

\begin{enumerate}[(1)]

\item Because $\norm{u*_s v-u*_s 0}=\norm{v}$ for an arbitrary control $v$ and by looking at step (\ref{thm:accessiblesetandendpointmap1}) in the proof of Theorem \ref{thm:accessiblesetandendpointmap} we see that there is a neighbourhood $V$ of $0$ and some $s$ close to one such that $u*_s V \subset U$. Then also 
$$\End_y(V) = \End_x(u*_s V) \subset \End_x(U).$$ 
It is sufficient to show that $\End_y(V)$ is a neighbourhood of $y$. 

\item For $\epsilon >0$ let $r(\epsilon)$ denote the maximal rank of $\End_y$ on the set $B(\epsilon,0)$ in $L^1([0,1];\mathbb R^k)$. $r$ is monotone and non-zero, so there is $\epsilon_0>0$ such that $r(\epsilon)=r(\epsilon_0)$ for all $\epsilon < \epsilon_0$. We can choose $\epsilon_0$ even smaller so that additionally $B(\epsilon_0,0)\subset V$. Let $V'$ denote the set of controls in $B(\epsilon_0,0)$ where $\End_y$ has maximal rank $r(\epsilon_0)$. Then $V'$ is open. $B(\epsilon/2,0)$ contains a control $v$ where $\End_y$ has rank $r(\epsilon_0)$. Then $v*\check v \in V'$, and $y=\End_y(v*\check v) \in \End_y(V')$. It is sufficient to show that $\End_y(V')$ is neighbourhood of $y$. 

\item $\End_y$ has constant rank on $V'$, so by the Constant Rank Theorem \ref{thm:constantranktheorem} there is an even smaller neighbourhood $V''$ of $v*\check v$ such that $\End_y(V'')$ is a small submanifold $L$. It is sufficient to show that $L$ is a neighbourhood of $y$.

\item By the same arguments as in step (\ref{thm:accessiblesetandendpointmap2}) in the proof of Theorem \ref{thm:accessiblesetandendpointmap} the vectorfields $X_j$ are tangent to $L$, and so are all the brackets of  $X_j$. Because of the bracket generating condition, the dimension of $L$ equals the dimension of $M$, and therefore $L$ is an open neighbourhood of $y$. \qedhere

\end{enumerate}
\end{proof}

\subsubsection{Proof of Chow's Theorem}

\begin{theorem}[Chow's Theorem] \index{Chow's Theorem}
Let $H$ be a locally finitely generated, bracket generating distribution on a manifold $M$. Then any two points in $M$ can be joined by a horizontal path. 
\end{theorem}

\begin{proof}
First, we will work locally, and therefore we can assume that $M$ is a connected, open subset of $\mathbb R^n$, and that $H$ is spanned by $X_1, \dots, X_k \in \mathfrak X(M)$. By making $M$ small enough we can assure that assumptions (\ref{endpointmap:assumption1}) to (\ref{endpointmap:assumption4}) of Section \ref{sec:endpointmap} hold. So we can apply Theorem \ref{thm:accessiblesetandendpointmap} which asserts that for any $x \in M$, $Acc(x)$ equals the image of $\End_x$, which is an initial submanifold. Horizontal vectorfields are tangent to it, and so are all their brackets. Because of the bracket generating condition its dimension equals the dimension of $M$, and therefore it is open. 

Now we return to the global case. By the previous considerations, every point in $M$ has a horizontally pathwise connected neighbourhood. Then by Lemma \ref{lem:locallyhorizontallypathwise} $M$ is horizontally pathwise connected.
\end{proof}

%% file: chapters/poincarelemma.tex
%
%

\section{The Poincar\'e Lemma}\label{sec:poincare}

The Poincar\'e Lemma in its general form states that locally, all closed differential forms are exact. For one-forms, this results to the following lemma:

\begin{lemma} \index{Poincar\'e Lemma}
Let $U$ be a star domain in $\mathbb R^n$. Let $h_i:U \rightarrow \mathbb R$ be smooth functions such that $\frac{\partial h_i}{\partial x_j}=\frac{\partial h_j}{\partial x_i}$ for all $1 \leq i,j \leq n$. Then there exists a smooth function $f:U \rightarrow \mathbb R$ such that $h_i = \frac{\partial f}{\partial x_i}$.
\end{lemma}

In the above situation $f$ can be recovered from the functions $h_i$ by integration. Indeed, if $\gamma:[0,1] \rightarrow \mathbb R^n$ is an absolutely continuous curve, $\gamma(0)=x_0$, $\gamma(1)=x$ and $u_i=\frac{\ud \gamma_i}{\ud t}$, then
$$f(x)=f(x_0)+ \int_0^1 \frac{\ud (f \circ \gamma)}{\ud t}(t)\ \ud t = \int_0^1 \sum_{i=1}^n u_i(t) h_i(\gamma(t))\ \ud t.$$

It is possible to adapt this idea to subriemannian geometry. If $H$ is a bracket generating distribution, then by Chow's Theorem \ref{thm:chow} the points $x$ and $x_0$ can be connected by a horizontal path $\gamma$. Using the same procedure as above, $f$ can be reconstructed from its horizontal derivatives alone. Indeed, if $\gamma$ is controlled by horizontal vectorfields $X_1, \dots, X_k$ with controls $u_i, \dots, u_k \in L^1([0,1];\mathbb R)$, then
\begin{equation}\tag{$*$} \label{eqn:reconstructf}
f(x)=f(x_0)+ \int_0^1 \frac{\ud (f \circ \gamma)}{\ud t}(t)\ \ud t = \int_0^1 \sum_{i=1}^k u_i(t) (X_i f)(\gamma(t))\ \ud t.
\end{equation}
In the following we will study how continuity and differentiability is passed on from the horizontal derivatives of $f$ to $f$. 

\begin{lemma}\label{lem:sublipschitz}
Let $H$ be a regular bracket generating distribution endowed with a smooth subriemannian metric. Let $f:M \rightarrow \mathbb R$ be a function such that 
$$X f : M \rightarrow \mathbb R, \quad (X f)(x) := \evaluate{\frac{\ud}{\ud t}}{0} f \circ Fl^X_t(x)$$
exists and is continuous for all $X \in \mathfrak X_{loc,H}(M)$. Then for every point $x_0 \in M$ there is a neighbourhood $U$ of $x_0$ and a constant $C$ such that the following statement holds for all $x \in U$:
$$\abs{f(x)-f(x_0)} \leq C d(x,x_0).$$
Here $d(x,x_0)$ is the subriemannian distance of $x$ and $x_0$ as in Definition \ref{def:subriemanniandistance}.
\end{lemma}

\begin{proof}
This is a local statement, so we can assume that there are everywhere defined horizontal vectorfields $X_1, \dots, X_k$ spanning $H$. Because $H$ is regular, we can assume additionally that $X_1, \dots X_k$ are orthonormal. Then for all horizontal curves $\gamma$ that connect $x_0$ to $x$ by means of $L^1$-controls $u_i$ we have
\begin{equation*}\begin{split}
\abs{f(x)-f(x_0)} \leq & \int_0^1 \sum_{i=1}^k \abs{u_i(t)} \abs{ (X_i f)(\gamma(t))} \ \ud t \\
\leq & \int_0^1 \left( \sum_{i=1}^k u_i(t)^2 \right)^{1/2} \left( \sum_{i=1}^k (X_i f)(\gamma(t))^2 \right)^{1/2} \ud t \\
\leq & \ C \int_0^1 \left( \sum_{i=1}^k u_i(t)^2 \right)^{1/2} \ud t = C \int_0^1 \norm{\dot\gamma(t)} \ \ud t = C \cdot l(\gamma), 
\end{split}\end{equation*}
where $C$ bounds $(\sum_{i=1}^k (X_i f)^2)^{1/2}$. Therefore 
\begin{equation*}
\abs{f(x)-f(x_0)} \leq C \inf_\gamma \ l(\gamma) = C d(x,y). \qedhere
\end{equation*}
\end{proof}

It is important to note that $\abs{f(x)-f(y)} \leq C d(x,y)$ does not imply that $f$ is Lipschitz continuous with respect to the Euclidean structure in a chart. This will be illustrated by the following example. 

\begin{example} 
We look at the function
$$f(x,y,z)=\left(\abs{x}^{5/2}+\abs{y}^{5/2}+\abs{z}^{3/2} \right)^{1/2}$$
on the Heisenberg geometry, i.e. $\mathbb R^3$ endowed with the distribution $H$ spanned by the vectorfields 
$$X_1=\frac{\partial}{\partial x}-\frac{y}{2} \frac{\partial}{\partial z} \quad \text{ and } \quad X_2=\frac{\partial}{\partial y}+\frac{x}{2} \frac{\partial}{\partial z}.$$
From $[X_1,X_2]=\frac{\partial}{\partial z}$ we see that $H$ is bracket generating. We calculate
\begin{align*}
\frac{\partial f}{\partial x}=
\begin{cases}
\displaystyle \frac{5}{4} \left(\abs{x}^{5/2}+\abs{y}^{5/2}+\abs{z}^{3/2} \right)^{-1/2} \abs{x}^{3/2} \sgn(x) & \text{ for } (x,y,z) \neq (0,0,0) \\
0 & \text{ for } (x,y,z) = (0,0,0)
\end{cases}\\
\frac{\partial f}{\partial y}=
\begin{cases}
\displaystyle \frac{5}{4} \left(\abs{x}^{5/2}+\abs{y}^{5/2}+\abs{z}^{3/2} \right)^{-1/2} \abs{y}^{3/2} \sgn(y) & \text{ for } (x,y,z) \neq (0,0,0) \\
0 & \text{ for } (x,y,z) = (0,0,0)
\end{cases}\\
\frac{\partial f}{\partial z}=
\begin{cases}
\displaystyle \frac{3}{4} \left(\abs{x}^{5/2}+\abs{y}^{5/2}+\abs{z}^{3/2} \right)^{-1/2} \abs{z}^{1/2} \sgn(z) & \text{ for } (x,y,z) \neq (0,0,0) \\
\pm \infty & \text{ for } (x,y,z) = (0,0,0).
\end{cases}
\end{align*}
Because of the singularity of $\frac{\partial f}{\partial z}$ at zero, $f$ is not Lipschitz continuous with respect to the Euclidean structure. However, the functions 
\begin{equation*}
\frac{\partial}{\partial x} f, \quad \frac{\partial}{\partial y} f, \quad 
\left(-\frac{y}{2} \frac{\partial}{\partial z}\right) f \quad \text{and} \quad \left(\frac{x}{2} \frac{\partial}{\partial z}\right) f
\end{equation*}
are continuous, as the following estimates show. We will denote the $p$-norm on $\mathbb R^3$ by $\norm{\cdot}_p$, and we will use the fact that all norms on $\mathbb R^3$ are equivalent. 
\begin{align*}
\abs{\frac{\partial f}{\partial x}} & \leq \frac{5}{4} \left(\abs{x}^{5/2}+\abs{y}^{5/2}+\abs{z}^{5/2} \right)^{-1/2} \abs{x}^{3/2}\\
 & \leq \frac{5}{4} ( \norm{(x,y,z)}_{5/2} )^{-5/4} ( \norm{(x,y,z)}_1 )^{3/2} 
     \leq \mathrm{const} ( \norm{(x,y,z)}_1 )^{1/4}.\\
\abs{\frac{\partial f}{\partial y}} & \leq \frac{5}{4} \left(\abs{x}^{5/2}+\abs{y}^{5/2}+\abs{z}^{5/2} \right)^{-1/2} \abs{y}^{3/2}\\
 & \leq \frac{5}{4} ( \norm{(x,y,z)}_{5/2} )^{-5/4} ( \norm{(x,y,z)}_1 )^{3/2} 
     \leq \mathrm{const} ( \norm{(x,y,z)}_1 )^{1/4}.\\
\abs{\frac{y}{2} \frac{\partial f}{\partial z}}
 & \leq \abs{\frac{y}{2}} \frac{3}{4} \left(\abs{x}^{5/2}+\abs{y}^{5/2}+\abs{z}^{5/2} \right)^{-1/2} \abs{z}^{1/2}\\
 & \leq \frac{3}{8} \norm{(x,y,z)}_1 (\norm{(x,y,z)}_{5/2})^{-5/4} (\norm{(x,y,z)}_1)^{1/2} 
     \leq \mathrm{const} (\norm{(x,y,z)}_1)^{1/4}.\\
\abs{\frac{x}{2} \frac{\partial f}{\partial z}}
 & \leq \abs{\frac{x}{2}} \frac{3}{4} \left(\abs{x}^{5/2}+\abs{y}^{5/2}+\abs{z}^{5/2} \right)^{-1/2} \abs{z}^{1/2}\\
 & \leq \frac{3}{8} \norm{(x,y,z)}_1 (\norm{(x,y,z)}_{5/2})^{-5/4} (\norm{(x,y,z)}_1)^{1/2} 
     \leq \mathrm{const} (\norm{(x,y,z)}_1)^{1/4}.
\end{align*}
Therefore also $X_1 f$ and $X_2 f$ are continuous. Because any horizontal vectorfield $X$ can be written as $X=f_1 X_1+f_2 X_2$ with smooth functions $f_1$ and $f_2$, $X f$ is continuous for all horizontal vectorfields $X$. So the conditions of Lemma \ref{lem:sublipschitz} are satisfied, and the lemma implies that $f$ is Lipschitz continuous with respect to the subriemannian distance. But, as we have seen above, $f$ is not Lipschitz continuous with respect to the Euclidean structure. Note that also the conditions of Theorem \ref{thm:mytheorem} are satisfied with $r=0$, and that the theorem implies the continuity of $f$. 
\end{example}

\begin{remark}
Let us suppose that $M$ is an $n$-dimensional subriemannian manifold with distribution $H$, and that $k$ brackets are needed to express the gradient $(\frac{\partial}{\partial x^1}, \dots, \frac{\partial}{\partial x^n})$ in terms of horizontal vectorfields, i.e.~$k$ is the smallest number such that $H_k=TM$. For some $r \geq k$, let $f:M \rightarrow \mathbb R$ be a function such that 
$$X f : M \rightarrow \mathbb R, \quad (X f)(x) := \evaluate{\frac{\ud}{\ud t}}{0} f \circ Fl^X_t(x)$$
exists and is of class $C^r$ for all $X \in \mathfrak X_{loc,H}(M)$. Then the gradient $(\frac{\partial f}{\partial x^1}, \dots, \frac{\partial f}{\partial x^n})$ of $f$ is of class $C^{r-k}$, and therefore $f$ is of class $C^{r-k+1}$. However, the following theorem yields a stronger result, namely that $f$ is of class $C^r$ no matter how many brackets are needed to span the tangent space. 
\end{remark}

\begin{theorem}\label{thm:mytheorem}
Let $H$ be a locally finitely generated, bracket generating distribution on a manifold $M$, and $r \geq 0$. If $f:M \rightarrow \mathbb R$ is a function such that 
$$X f : M \rightarrow \mathbb R, \quad (X f)(x) := \evaluate{\frac{\ud}{\ud t}}{0} f \circ Fl^X_t(x)$$
exists and is of class $C^r$ for all $X \in \mathfrak X_{loc,H}(M)$, then $f$ is also of class $C^r$. 
\end{theorem}

\begin{proof}
We will show that $f$ is of class $C^r$ at every $x_0 \in M$. This is a local question, so we can assume that conditions (\ref{endpointmap:assumption1}) to (\ref{endpointmap:assumption4}) of Section \ref{sec:endpointmap} are satisfied, so that the endpoint map is well-defined. 

\begin{enumerate}[(1)]

\item First, we will construct a smooth local section of the endpoint map. That is a smooth map $u$ defined near $x_0$ such that $\End_{x_0}(u (x)) =x$. By Lemma \ref{lem:normallyaccessible} there is a normal control $u_0$ such that $\End_{x_0}(u_0)=x_0$. By Theorem \ref{thm:endpointmapisopen} $\End_{x_0}$ is open and the rank of $\End_{x_0}$ at $u_0$ equals the dimension of $M$, say $n$. $\End_{x_0}$ has constant rank near $u_0$, so by the Constant Rank Theorem \ref{thm:constantranktheorem} there is a subspace $E \subset L^1([0,1];\mathbb R^k)$, an open neighbourhood $U$ of $u_0$, an open neighbourhood $V$ of $x$ and charts 
\begin{alignat*}{2}
\phi: U & \subset L^1([0,1];\mathbb R^k) && \rightarrow \phi(U) \subset \mathbb R^n \times E, \\
\psi: V & \subset M                                  && \rightarrow \psi(V) \subset \mathbb R^n
\end{alignat*}
such that for all $(t^1, \dots, t^n,e) \in \phi(U)$ we have
$$\psi \circ \End_{x_0} \circ \phi(t^1, \dots, t^n,e)=(t^1, \dots, t^n).$$
Without loss of generality we can assume that $\phi(u_0)=0$. Then the map 
\begin{equation*}
u: \left\{ \begin{split}
\psi^{-1}(\{ (t^1, \dots, t^n): (t^1,\dots,t^n,0)\in \phi(U) \}) & \rightarrow L^1([0,1];\mathbb R^k)\\
x & \mapsto \phi^{-1}(\psi(x),0)
\end{split}\right.
\end{equation*}
has the desired properties. 

\item Let $\Gamma$ be the operator assigning to a control in the domain of $\End_{x_0}$ the unique curve $\gamma$ with this control and $\gamma(0)=x_0$. By Theorem \ref{thm:controlledcurve} $\Gamma$ is smooth as a map into $C([0,1];M)$. 

\item By Theorem \ref{thm:smoothsuperpositionoperator} about superposition operators and because $X_i f$ is of class $C^r$, the map 
$$C([0,1];M) \rightarrow C([0,1];\mathbb R), \quad \gamma \mapsto (X_i f) \circ \gamma$$ is of class $C^r$. 

\item Let $I$ denote the map 
$$L^1([0,1];\mathbb R^k) \times C([0,1];\mathbb R^k) \rightarrow \mathbb R, \quad (u, h) \mapsto \int_0^1 \sum u_i(t) h_i(t) \ud t.$$
$I$ is smooth because it is linear and bounded. 

\item From equation \eqref{eqn:reconstructf}, using the notation $X=(X_1, \dots, X_k)$, we have
$$f=f(x_0) + I \left(u, (X f) \circ \Gamma \circ u\right).$$
Therefore $f$ is of class $C^r$. \qedhere
\end{enumerate}
\end{proof}